\documentclass[10pt]{article}
\usepackage{amsmath,amssymb,amsthm}
\usepackage{graphicx,subfigure,float,url,color}
\usepackage{mathrsfs}
\usepackage[colorlinks=true]{hyperref}
\usepackage{pdfsync}
\usepackage{enumitem,dsfont}

\def\R{\mathrm{I\kern-0.21emR}}
\def\N{\mathrm{I\kern-0.21emN}}

\newcommand{\Lip}{\operatorname{Lip}}

\newcommand{\supp}{\mathrm{supp}}

\newcommand{\diam}{\mathrm{diam}}
\newcommand{\cl}{\mathrm{cl}}

\renewcommand{\geq}{\geqslant}
\renewcommand{\leq}{\leqslant}

\newtheorem{theorem}{Theorem}  
\newtheorem{proposition}{Proposition}

\newtheorem{lemma}{Lemma}

\theoremstyle{definition}\newtheorem{remark}{Remark}

\title{Universal approximations of quasilinear PDEs by finite distinguishable particle systems}

\author{
Thierry Paul\footnote{CNRS Laboratoire Ypatia des Sciences Mathématiques LYSM, Rome, Italy (\texttt{thierry.paul@sorbonne-universite.fr}).}
\and
Emmanuel Tr\'elat\footnote{Sorbonne Universit\'e, Universit\'e Paris Cit\'e, CNRS, Inria, Laboratoire Jacques-Louis Lions, LJLL, F-75005 Paris, France (\texttt{emmanuel.trelat@sorbonne-universite.fr}).}
}

\begin{document}

\maketitle

\begin{abstract}
In this paper, we prove that sufficiently regular solutions of any quasilinear PDE can be approximated by solutions of systems of $N$ distinguishable particles, to within $1/\ln(N)$.
This intruiguing result is related to recent developments in graph limit theory.
\end{abstract}

\tableofcontents

\section{Introduction}\label{sec_intro}

Particle approximations are well known for some classes of PDEs, like fluid equations: for fluid Euler or Navier-Stokes equations, one often speaks of ``fluid particles", in accordance with the classical Eulerian or Lagrangian viewpoints. 

In this paper we prove that sufficiently regular solutions of \emph{any} quasilinear evolution equation can be approximated by solutions of systems of $N$ distinguishable particles, to within $1/\ln\ln(N)$. 

To give a flavor of this surprising universal result and to simplify the setting, in this introduction let us restrict to the linear case. 
Let $n,d\in\N^*$ and let $\Omega$ be an open bounded subset of $\R^n$.
We consider a linear evolution equation
\begin{equation}\label{linearpde}
\partial_t y(t,x) = Ay(t,x)
\end{equation}
where $y(t,x)\in\R^d$ for $t>0$ and $x\in\Omega$, and where $A:D(A)\rightarrow L^2(\Omega,\R^d)$ is a linear operator generating a $C_0$ semigroup and $\Omega$ is a domain of $\R^n$. Let $[A]$ be the (distributional) Schwartz kernel of $A$, so that $(Ag)(x) = \langle[A](x,\cdot),f\rangle$ for every $g\in\mathscr{C}^\infty_c(\Omega,\R^d)$, in the distributional sense.
For instance if $[A](x,x')=\delta_x'$, the distributional derivative of the Dirac measure $\delta_x$ at $x$, then $A=-\partial_x$.

Now, in a first step, we approximate (for example, by convolution) the distribution $[A]$ with a family of smooth kernels $\sigma_\varepsilon$, depending on a parameter $\varepsilon>0$, such that $\sigma_\varepsilon$ converges in the distributional sense to $[A]$ as $\varepsilon\rightarrow 0$.
Now, for every $\varepsilon>0$, we have a bounded operator $A_\varepsilon$ on $L^2(\Omega)$, given by $(A_\varepsilon g)(x) = \int_\Omega \sigma_\varepsilon(x,x') g(x')\, dx'$ for every $g\in\mathscr{C}^\infty_c(\Omega,\R^d)$, and we consider the linear evolution equation
\begin{equation}\label{linearpde_epsilon}
\partial_t y_\varepsilon(t,x) = A_\varepsilon y_\varepsilon(t,x) = \int_\Omega \sigma_\varepsilon(x,x') y_\varepsilon(t,x')\, dx' ,
\end{equation}
which is viewed as an approximation of the linear evolution equation \eqref{linearpde}. Taking $y_\varepsilon(0,\cdot)=y(0,\cdot)$, as one can expect, under general assumptions we have, on a compact interval, $\Vert y_\varepsilon(t)-y(t)\Vert_{L^2} \leq C\varepsilon e^{\omega t}$ for some $C>0$ and $\omega\in\R$ not depending on the solutions -- or, some similar estimate in $\varepsilon$; in any case, something small as $\varepsilon$ is small.

In a second step, we approximate the integral evolution equation \eqref{linearpde_epsilon} by a family of finite particle systems, indexed by $N$, defined by
\begin{equation}\label{linearpde_system_epsilon}
\dot\xi_{\varepsilon,i}^N(t) = \frac{1}{N}\sum_{j=1}^N \sigma_\varepsilon (x_i^N,x_j^N) \, \xi_{\varepsilon,j}^N(t) .
\end{equation}
This can be done simply by approximating the integral in \eqref{linearpde_epsilon} with a Riemann sum, performed on a partition $\Omega=\cup_{i=1}^N\Omega_i^N$ of the domain $\Omega$.
Standard estimates allow one to measure the discrepancy between solutions of \eqref{linearpde_epsilon} and finite particle solutions of \eqref{linearpde_system_epsilon}. 

Applying the triangular inequality, we obtain an estimate of the form
$$
\bigg\Vert y(t,\cdot) - \sum_{i=1}^N \xi_{\varepsilon,i}^N(t) \, \mathds{1}_{\Omega_i^N}(\cdot) \bigg\Vert_{L^2(\Omega,\R^d)}
\leq \ C\left( \varepsilon + \frac{e^{C/\varepsilon}}{N} \right) 
$$
on a compact interval. Finally, choosing $\varepsilon \sim \frac{1}{\ln\ln N}$, we get the estimate
\begin{equation}\label{intro_loglog}
\bigg\Vert y(t,\cdot) - \sum_{i=1}^N \xi_{\varepsilon,i}^N(t) \, \mathds{1}_{\Omega_i^N}(\cdot) \bigg\Vert_{L^2(\Omega,\R^d)}
\leq \frac{C}{\ln N} 
\end{equation}
(or maybe, a positive power of the right-hand side) on a compact interval. 
This shows that, under few assumptions and in whole generality, any linear PDE can be approximated by an explicit family of finite particle systems, to within $1/\ln(N)$. 

%

This surprising result has actually emerged as a byproduct of our article \cite{PaulTrelat}, in which we have revisited and extended various ways to pass to the limit in finite systems of (possibly distinguishable, deterministic) particles as the number of particles tends to infinity. In particular, we have shown that the continuum / graph limit equation, which is derived as above from the particle system by the Riemann sum theorem, can always be identified with the Euler equation derived by hydrodynamic considerations from the Vlasov equation that is itself obtained by mean field limit. This actually deep fact has led us to introduce a set of $x\in\Omega$ standing for ``labels" of the particles, making them distinguishable. 
Then the idea outlined above emerged, because in the PDE setting, $\Omega$ can naturally be the domain of an unbounded operator and the label $x\in\Omega$ is the spatial variable of the PDE. 

In this study, an inspiring model has been the linear Hegselmann--Krause first-order consensus system (see \cite{HegselmannKrause})
\begin{equation}\label{HK_particle}
\dot\xi^N_i(t) = \frac{1}{N} \sum_{j=1}^N \sigma_{ij}^N \big(\xi^N_j(t)-\xi^N_i(t)\big), \qquad i\in\{1,\ldots,N\},
\end{equation}
which models propagation of opinions.
In \eqref{HK_particle}, the coefficient $\sigma_{ij}^N$ stands for the interaction of the agent $i$ with the agent $j$. The $N\times N$ matrix consisting of those coefficients, not necessarily symmetric, is naturally associated to a graph. This interpretation has certainly motivated the terminology of ``graph limit" (coined in \cite{Medvedev_SIMA2014}, used to describe how to pass to the continuum limit in \eqref{HK_particle} and get the evolution equation $\partial_t y(t,x) = \int_{\Omega} \sigma(x,x') (y(t,x')-y(t,x)) \, dx'$, whenever a limit function $\sigma$ (called ``graphon") does exist. Dropping the term $\xi^N_i(t)$ at the right-hand side of \eqref{HK_particle}, one can see that \eqref{HK_particle} coincides with \eqref{linearpde_system_epsilon} and provides a good and inspiring model to approach any linear evolution equation. 

To our knowledge, an early existing article in which a related (though different) idea can be found is \cite{LionsMas-Gallic}, in which the authors provide a deterministic particle approximation for heat equations and porous media equations. See also \cite{CarrilloCraigPatacchini_CVPDE2019, dif0, EspositoPatacchiniSchlichtingSlepcev_ARMA2021, FigalliKang_APDE2019,Carrillo Esposito Wu CVPDE 2024}


%

{

\subsection{Setting}
Multi-agent collective models have regained an increasing interest over the last years, due in particular to their connection with mean field and graph limit equations. At the microscopic scale, such models consist of considering particles or agents evolving according to the dynamics
 \begin{equation}\label{eqagent}
\dot\xi_i(t) = \frac{1}{N} \sum_{j=1}^N G_{ij}^N(t,\xi_i(t),\xi_j(t)) ,
\qquad i\in\{1,\ldots,N\},
\end{equation}
for some (large) number of agents $N\in\N^*$ where, for every $i\in\{1,\ldots,N\}$, $\xi_i(t)\in\R^d$ (for some $d\in\N^*$) stands for various parameters describing the behavior of the $i^\textrm{th}$ agent and $G_{ij}^N:\R\times\R^d\times\R^d\rightarrow\R^d$ is a mapping modeling the interaction between the $i^\textrm{th}$ and $j^\textrm{th}$ agents.

This paper deals somehow with the  path inverse to the mean field limit one: given a general partial differential equation (PDE), is is possible to construct explicitly an agent system of the form \eqref{eqagent} such that the corresponding graph limit equation coincides with the given PDE we started with? Surprisingly, this happens to be true in a very general setting and this is a consequence of the analysis done in the paper.

The question of whether some classes of PDEs are a ``natural" limit of particle systems is classical in fluid mechanics and certainly dates back to Euler: it is classical that the Euler fluid equation can be seen, at least formally, as the limit of evolving ``particles of fluids". This has been formalized in the famous article \cite{Arnold} where Arnol'd interpreted the Euler equation as a geodesic equation in the space of diffeomorphisms, leading to a number of subsequent studies; we refer to \cite{Bauer} (see also the references therein) for a survey on how to ``cook up" appropriate groups of diffeomorphisms (and thus, of particle systems) to generate classes of fluid PDEs, like Euler, Camassa-Holm, etc.
We also refer to \cite{Gallagher_BAMS2019} for a recent survey on microscopic, mesoscopic and macroscopic scales for fluid dynamics.

But it is much less classical to show that other, more general PDEs can as well be obtained by passing to the limit in some particle systems.
For transport equations, the topic has been extensively studied in \cite{dif0,dif1,dif2}. 
Recently, thanks to the concept of graph limit elaborated in \cite{Medvedev_SIMA2014}, it has been possible to show that heat-like equations can as well be obtained as limits of particle systems (see also \cite{AyiPouradierDuteil_JDE2021, BiccariKoZuazua_M3AS2019, BonnetPouradierDuteilSigalotti_2021, EspositoPatacchiniSchlichtingSlepcev_ARMA2021}).
In \cite{FigalliKang_APDE2019}, the authors provide a rigorous derivation from the kinetic Cucker-Smale model to the macroscopic pressureless Euler system by hydrodynamic limit, using entropy methods and deriving error estimates. 

Actually, during the \emph{Le\c{c}ons Jacques-Louis Lions} given in our laboratory in the fall 2021 by Dejan Slepcev, we were intrigued by his way of deriving heat-like equations from unusual particle systems, by taking not only the limit as the number $N$ of agents tends to $+\infty$, but also another parameter $\varepsilon$ tends to $0$, at some precise scaling (see \cite{EspositoPatacchiniSchlichtingSlepcev_ARMA2021}). The role of $\varepsilon$ is to smoothen the dynamics. His striking exposition has been for us a great source of inspiration and has motivated the last part of the present article.

In this last part, we provide for a large range of quasilinear PDEs a natural and constructive way for associating an agent system to them. Shortly, as a particular case of our analysis, considering a general PDE
\begin{equation}\label{eqpde}
\partial_ty(t,x)=\sum_{\vert\alpha\vert\leq p} a_\alpha(t,x,y(t,x)) \partial_x^\alpha y(t,x) = A(t,x,y(t,x))y(t,x),
\end{equation}
we show that \eqref{eqpde} is the graph limit of (for example) the particle system \eqref{eqagent} with
$$
G_{ij}(t,\xi,\xi')=G_\varepsilon(t,i,j,\xi,\xi')=\xi'\sum_{\vert\alpha\vert\leq p} a_\alpha(t,x,\xi)\partial_{x'}^\alpha \frac{e^{-\frac{(x-x')^2}{2\varepsilon}}}{(\pi\varepsilon)^{\frac 12}}
$$ 
in the limit $N\gg\varepsilon^{-1}\rightarrow+\infty$, with some appropriate scalings. We establish convergence estimates in Wasserstein distance in general, and in $L^2$ norm under an additional (but general) semigroup assumption.

Finally, as announced, as a surprising byproduct built on the previous developments, we show in Section \ref{sec_PDE} that general quasilinear PDEs can be obtained by passing to the limit in explicit particle systems, thanks to two asymptotic parameters.

}

\subsection{Some tools and notations}

Let $(\Omega,\mathrm{d}_\Omega)$ be a complete metric space, endowed with a probability measure  $\nu\in\mathcal{P}(\Omega)$. 

\paragraph{Tagged partitions, and Riemann sum theorem.}
We say that $(\mathcal{A}^N,x^N)_{N\in\N^*}$ is a family of \emph{tagged partitions} of $\Omega$ associated with $\nu$ if 
$\mathcal{A}^N=(\Omega^N_1,\ldots,\Omega^N_N)$ is a $N$-tuple of disjoint subsets $\Omega^N_i\subset\Omega$ such that \begin{equation}\label{def_tagged}
\Omega = \bigcup_{i=1}^N\Omega^N_i\qquad \textrm{with}\qquad
\nu(\Omega^N_i)=\frac{1}{N}
\quad \textrm{and}\quad
\diam_\Omega(\Omega^N_i)\leq \frac{C_\Omega}{N^\gamma}\qquad \forall i\in\{1,\ldots,N\},
\end{equation}
for some $C_\Omega>0$ and $\gamma>0$ not depending on $N$,
and $x^N=(x^N_1,\ldots,x^N_N)$ is a $N$-tuple of points $x^N_i\in\Omega^N_i$. 
Here, $\diam_\Omega(\Omega^N_i)$ is the supremum of all $\mathrm{d}_\Omega(x,x')$ over all possible $x,x'\in\Omega^N_i$.

Families of tagged partitions always exist when $\Omega$ is a compact $n$-dimensional smooth manifold having a boundary or not and $\nu$ is a Lebesgue measure on $\Omega$, with $\gamma=1/n$. 
For instance, when $\Omega=[0,1]$, we take $\Omega^N_i=[a^N_i,a^N_{i+1})$ for some subdivision $0=a^N_1<a^N_2<\cdots<a^N_{N+1}=1$ satisfying \eqref{def_tagged}; when $d\nu(x)=dx$, a natural choice is $a^N_i=\frac{i-1}{N}$, and $x^N_i=a^N_i$ or $\frac{a^N_i+a^N_{i+1}}{2}$, for every $i\in\{1,\ldots,N\}$ (and then $C_\Omega=1$ and $r=1$ in this case). When $\Omega$ is a compact domain of $\R^n$, a family of tagged partitions is obtained by considering a family of meshes, as classically done in numerical analysis, with $\gamma=1/n$.

The concept of tagged partition is used in Riemann (and more generally, Henstock-Kurzweil) integration theory. 
We refer to \cite{Fremlin} for (much more) general results.
A real-valued function $f$ on $\Omega$, of compact support, is said to be $\nu$-Riemann integrable if it is bounded, $\nu$-measurable, and if, for any family $(\mathcal{A}^N,x^N)_{N\in\N^*}$ of tagged partitions, we have
\begin{equation}\label{CV_Riemann_sum_1}
\sum_{i=1}^N \int_{\Omega_i^N} \vert f(x)-f(x^N_i) \vert\, d\nu(x) = \mathrm{o}(1) 
\end{equation}
and thus
\begin{equation}\label{CV_Riemann_sum_2}
\int_{\Omega} f\, d\nu = \frac{1}{N} \sum_{i=1}^N f(x^N_i) + \mathrm{o}(1)
\end{equation}
as $N\rightarrow +\infty$.
A function $f$ of essential compact support on $\Omega$ is $\nu$-Riemann integrable if and only if $f$ is bounded and continuous $\nu$-almost everywhere on $\Omega$. Given any Lipschitz continuous function $f$ on $\Omega$, of Lipschitz constant $\Lip(f)$, we have
\begin{equation}\label{CV_Riemann_sum_3}
\left\vert \int_{\Omega} f\, d\nu - \frac{1}{N} \sum_{i=1}^N f(x^N_i) \right\vert \leq \frac{C_\Omega \Lip(f)}{N^\gamma}
\end{equation}
for any $N\in\N^*$ (see, e.g., \cite[Appendix A.4.1]{PaulTrelat} for this very standard and elementary result).

\paragraph{Functional spaces.}
Let $d\in\N^*$. We denote by $\Vert\ \Vert_{\R^d}$ the Euclidean norm on $\R^d$.

We denote by $\Lip(\Omega,\R^d)$ the space of functions $f:\Omega\rightarrow\R^d$ that are Lipschitz continuous on $\Omega$, of Lipschitz constant 
$$
\Lip(f) = \sup_{{x,x'\in\Omega}\atop{x\neq x'}} \frac{\Vert f(x)-g(x')\Vert}{\mathrm{d}_\Omega(x,x')} . 
$$

For any $k\in[1,+\infty]$, we denote by $L^k_\nu(\Omega,\R^d)$, or simply by $L^k(\Omega,\R^d)$, the usual Lebesgue space. For $k<+\infty$, it consists of measurable functions $f:\Omega\rightarrow\R^d$ such that $\Vert f\Vert_{L^k} = \left(\int_\Omega\Vert f(x)\Vert_{\R^d}^k\, d\nu(x)\right)^{1/k}<+\infty$. For $k=+\infty$, it consists of $\nu$-essentially bounded measurable functions $f:\Omega\rightarrow\R^d$, endowed with the essential supremum norm $\Vert\ \Vert_{L^\infty}$. 
Given any $p\in\N$, we denote by $W^{p,k}_\nu(\Omega,\R^d)$, or simply $W^{p,k}(\Omega,\R^d)$, the Sobolev space of functions $f:\Omega\rightarrow\R^d$ whose partial (distributional) derivatives up to order $p$ are identified with functions of $L^k(\Omega,\R^d)$, endowed with the norm
$$
\Vert f\Vert_{W^{p,k}} = \max_{\vert\alpha\vert\leq p} \Vert D^\alpha g\Vert_{L^k}.
$$
For $k=2$, we denote $H^p(\Omega,\R^d)=W^{p,2}(\Omega,\R^d)$.

Assume that $\Omega$ is a smooth finite-dimensional manifold.
For any $k\in[1,+\infty]$, we denote by $\mathscr{C}^k(\Omega,\R^d)$ (resp., $\mathscr{C}^\infty_c(\Omega,\R^d)$) the set of functions $f:\Omega\rightarrow\R^d$ of class $\mathscr{C}^k$ (resp., moreover, of compact support).

\section{Finite particle approximations of quasilinear evolution systems}\label{sec_abstract}
Let $\Omega$ be a complete metric space, endowed with the distance $\mathrm{d}_\Omega$ and with a probability measure  $\nu\in\mathcal{P}(\Omega)$, assumed to have the property of having a family $(\mathcal{A}^N,x^N)_{N\in\N^*}$ of tagged partitions associated with $\nu$, satisfying \eqref{def_tagged}.

Let $d\in\N^*$ and let $X=\mathcal{F}(\Omega,\R^d)$ be a Banach space of functions on $\Omega$, taking their values in $\R^d$. 
Typical examples are Sobolev spaces, like $X=L^2(\Omega,\R^d)$ or more generally $X=H^s(\Omega,\R^d)$ for some $s\in\R$, the space of continuous functions $X=\mathscr{C}^0(\Omega,\R^d)$ or of functions of class $C^p$, the space of analytic functions, H\"older spaces, etc. 

Let $T>0$.
We consider the quasilinear evolution equation
\begin{equation}\label{abstract_quasilinear}
\boxed{
\dot{y}(t) = A[t,y(t)]\, y(t) + f[t,y(t)]
}
\end{equation}
for $t\in[0,T]$, with some initial condition $y(0)=y^0\in X$, where $A[t,z]$ is a linear operator on $X$, of domain $D(A[t,z])$, and $f[t,z]\in X$, for all $(t,z)\in[0,T]\times X$.


Assuming that $t\mapsto y(t)\in X$ is a solution of \eqref{abstract_quasilinear}, since $y(t)$, as an element of the functional space $X=\mathcal{F}(\Omega,\R^d)$, is a function on $\Omega$, in the sequel we denote indifferently $y(t)(x)=y(t,x)$ for all $t\geq 0$ and $x\in\Omega$.

Our objective is to prove that, under appropriate assumptions, sufficiently regular solutions $t\mapsto y(t)$ of \eqref{abstract_quasilinear} can be approximated by the solutions of a family of finite particle systems evolving in $\Omega$.

Note that we use brackets to denote $A[t,y(t)]$ and $f[t,y(t)]$ in \eqref{abstract_quasilinear}, in order to underline their possible nonlocal dependence with respect to $x\in\Omega$: $A[t,y(t)]$ and $f[t,y(t)]$ do not  necessarily depend only on $y(t,x)$, the value at $x$ of the function $y(t)\in X$ (for instance, a term like $(\int_\Omega y(t,x')^2\, dx' ) y(t,x)$, or like $\int_\Omega \rho(x-x')y(t,x')\, dx' \, y(t,x)^3$ as it is the case in some Vlasov equations).

\medskip

The section is structured as follows.

In Section \ref{sec_abstract_setting}, we recall the general setting introduced by Kato to ensure existence and uniqueness of solutions for abstract quasilinear evolution equations. 

In order to derive families of finite particles systems approximating the quasilinear evolution equation \eqref{abstract_quasilinear}, our strategy will be in two steps:
\begin{enumerate}[parsep=0mm, itemsep=0mm, topsep=0mm]
\item Given a family of \emph{bounded} linear operators $A_\varepsilon[t,z]$ approximating $A[t,z]$, derive a convergence estimates of solutions $y_\varepsilon$ of the $\varepsilon$-approximation evolution equation a given solution $y$ of \eqref{abstract_quasilinear}.
This is done in Section \ref{sec_abstract_approx}.
\item In Section \ref{sec_abstract_particle_approx}, for any $\varepsilon$ fixed, derive a family of finite particle systems, indexed by $N\in\N^*$, whose solutions converge as $N\rightarrow+\infty$ to solutions of the $\varepsilon$-approximation evolution equation. This convergence will follows from Theorem \ref{thm_estim_graph} in Appendix \ref{app_thm_estim_graph}. Of course, we will have to be careful enough to keep track of constants regarding the dependence with respect to $\varepsilon$ and $N$. 
\end{enumerate}
In Section \ref{sec_abstract_main_result}, we state the main result, establishing convergence of solutions of the finite particle system to a given solution of \eqref{abstract_quasilinear}. Estimates of convergence are obtained by the triangular inequality. They depend on $\varepsilon$ and $N$, but as already alluded the estimates blow up when $\varepsilon\rightarrow 0$ with $N$ being fixed and thus the limits must be taken at some appropriate scaling. 

In what follows, as a general notation, given any Banach space $E$, we denote by $\Vert\ \Vert_E$ its norm and by $L(E)$ the Banach space of all bounded linear operators on $E$. Given another Banach space $F$, the notation $L(E,F)$ stands for the Banach space of all bounded linear operators from $E$ to $F$.

\subsection{Abstract quasilinear evolution systems}\label{sec_abstract_setting}


In this section, as well as in the next Section \ref{sec_abstract_approx}, we do not need that $X$ be a space of functions: 
$X$ can be any arbitrary Banach space.

Existence and uniqueness of a solution of \eqref{abstract_quasilinear} are classical and are ensured by proving that the following mapping $\Phi$ is contractive and thus has a fixed point (see \cite{Kato_1975, Kato_1993} or \cite[Section 6.4, Theorem 4.6]{Pazy}): given an appropriate function $t\mapsto z(t)\in X$, $y(\cdot)=\Phi(z(\cdot))$ is defined as the unique solution of $\dot{y}(t) = A[t,z(t)]\, y(t) + f[t,z(t)]$ such that $y(0)=y^0$. This is done under the following classical assumptions, taken from \cite{Kato_1993, Sanekata_1989} (in fact, slightly more general assumptions are done in \cite{Kato_1993}).
The following setting and results are due to Kato.

\smallskip

\noindent\textbf{Functional space assumptions.}
\begin{enumerate}[label=$(H_{\arabic*})$, leftmargin=*, parsep=0.7mm, itemsep=0.7mm, topsep=0.7mm]
\item\label{H_Z} There exists a Banach subspace $Z$ of $X$, dense in $X$ and continuously embedded in $X$, i.e., there exists $C_1>0$ such that $\Vert z\Vert_X\leq C_1\Vert z\Vert_Z$ for every $z\in Z$. 
\item\label{H_S} There exists an operator $S\in L(Z,X)$ such that $\Vert z\Vert_Z = \Vert z\Vert_X + \Vert Sz\Vert_X$ (graph norm); equivalently, $S$ is a closed operator on $X$ of domain $D(S)=Z$.
\end{enumerate}

\smallskip

Let $y^0\in Z$. 
In what follows, given any $r>0$, we denote by $B_Z(y^0,r) = \{ z\in Z\ \mid\ \Vert z-y^0\Vert_Z\leq r\}$ the closed ball in $Z$ of center $y^0$ and radius $r$, and by $\cl_X({B}_Z(y^0,r))$ its closure in $X$. Note that $\cl_X({B}_Z(y^0,r))=B_Z(y^0,r)$ if $X$ and $Z$ are reflexive.

\smallskip

\noindent\textbf{Assumptions on the operator.}
There exists $r>0$ such that, for all $t\in[0,T]$ and $z\in B_Z(y^0,r)$:
\begin{enumerate}[label=$(H_{\arabic*})$, resume, leftmargin=*, parsep=0.7mm, itemsep=0.7mm, topsep=0.7mm]
\setcounter{enumi}{2}
\item\label{H_stab} (Semigroup and stability) The operator $A[t,z]$ generates a $C_0$ semigroup $(e^{sA[t,z]})_{s\geq 0}$ on $X$, and there exist $M\geq 1$ and $\omega\in\R$ such that, for every $k\in\N^*$, for all $s_1,\ldots,s_k\geq 0$ and all $0\leq t_1\leq\cdots\leq t_k\leq T$, one has
$\Vert e^{s_1A(t_1,z)} \cdots e^{s_kA(t_k,z)} \Vert_{L(X)} \leq M e^{(s_1+\cdots+s_k)\omega}$.\footnote{Note that the latter stability estimate holds true (with $M=1$) if $(e^{sA[t,z]})_{s\geq 0}$ is a semigroup of contractions.
}
\item\label{H_regZ} 
$Z\subset D(A[t,z])$, $A[t,z]\in L(Z,X)$ depends continuously on $t$, and there exists $C_4\geq 0$ such that
$\Vert A[t,z_1]-A[t,z_2]\Vert_{L(Z,X)}\leq C_4 \Vert z_1-z_2\Vert_X$ for all $t\in[0,T]$ and $z_1,z_2\in B_Z(y^0,r)$.
\item\label{H_intertwining} (Intertwining condition) $S A[t,z]  = A[t,z] S + B[t,z] S$ 
with $B\in \mathscr{C}^0([0,T]\times \cl_X({B}_Z(y^0,r)),L(X))$,
and there exists $C_5\geq 0$ such that $\Vert B[t,z]\Vert_{L(X)}\leq C_5$ for all $(t,z)\in [0,T]\times \cl_X({B}_Z(y^0,r))$.
\end{enumerate}

\smallskip

\noindent\textbf{Assumptions on $f$.}
Finally, we assume that:
\begin{enumerate}[label=$(H_{\arabic*})$, resume, leftmargin=*, parsep=0.7mm, itemsep=0.7mm, topsep=0.7mm]
\item\label{H_f_Z} $f\in\mathscr{C}^0([0,T]\times B_Z(y^0,r),Z)$ and there exists $C_6\geq 0$ such that $\Vert f[t,z]\Vert_Z\leq C_6$ for all $t\in[0,T]$ and $z\in B_Z(y^0,r)$.
\item\label{H_f_Lip} There exists $C_7\geq 0$ such that
$\Vert f[t,z_1]-f[t,z_2]\Vert_X \leq C_7\Vert z_1-z_2\Vert_X$ for all $t\in[0,T]$ and $z_1,z_2\in B_Z(y^0,r)$.
\end{enumerate}

\smallskip

\noindent\textbf{Evolution system.}
As proved in \cite{Kato_1993, Sanekata_1989} (see also \cite[Section 6.4]{Pazy}), under Assumptions \ref{H_Z} to \ref{H_intertwining}, for every $z(\cdot)\in\mathscr{C}^0([0,T],X)$ such that $z(t)\in B_Z(y^0,r)$ for every $t\in[0,T]$, there exists an \emph{evolution system} $(U_z(t,s))_{0\leq s\leq t\leq T}$ on $X$, i.e., a family of operators $U_z(t,s)\in L(X)$ depending continuously on $(t,s)$ and satisfying, for all $0\leq s\leq r\leq t\leq T$:
\begin{enumerate}[label=$(E_{\arabic*})$, leftmargin=*, parsep=0.7mm, itemsep=0.7mm, topsep=0.7mm]
\item\label{E_morphism} $U_z(t,s) = U_z(t,\tau) U_z(\tau,s)$ and $U_z(s,s) = \mathrm{id}_X$.
\item\label{E_UZ} $U_z(t,s)Z\subset Z$; 
\item\label{E_estim_exp} $\Vert U_z(t,s)\Vert_{L(X)} \leq Me^{\omega(t-s)}$ and $\Vert U_z(t,s)\Vert_{L(Z)} \leq Me^{\omega(t-s)}$ (where the value of $\omega$ has been increased if necessary);
\item\label{E_ode} $\partial_t U_z(t,s) = A[t,z(t)]\, U_z(t,s)$ and $\partial_s U_z(t,s) = -U_z(t,s)A[s,z(s)]$ on $Z$ (the derivatives exist in $L(Z,X)$).
\end{enumerate}

\begin{proposition}\label{prop_existence_uniqueness}
Let $y^0\in Z$. Under Assumptions \ref{H_Z} to \ref{H_f_Lip}, there exists a unique solution $y(\cdot)\in\mathscr{C}^0([0,T'],Z)\cap\mathscr{C}^1([0,T'],X)$ of \eqref{abstract_quasilinear} such that $y(0)=y^0$, for some $T'\in(0,T]$. Moreover, for every $t\in[0,T']$ one has $y(t)\in B_Z(y^0,r)$ and
\begin{equation}\label{duhamel_implicit}
y(t) = U_y(t,0)y^0 + \int_0^t U_y(t,s) f[s,y(s)]\, ds .
\end{equation}
\end{proposition}

Note that, in contrast to the usual Duhamel formula in the classical linear case, the formula \eqref{duhamel_implicit} is implicit in general because of the dependence with respect to $y$, see Remark \ref{rem_explicit} further. 

\begin{remark}\label{rem_prop_existence_uniqueness}
It follows from the proof given hereafter that the time $T'$ only depends on $y^0$ and on the spaces and various constants given through Assumptions \ref{H_Z} to \ref{H_f_Lip}: given some Banach spaces $X$ and $Z$ satisfying \ref{H_Z} and \ref{H_S}, some $y^0\in Z$, some $r>0$, $M\geq 1$, $\omega\in\R$ and some nonnegative constants $C_4,C_5,C_6,C_7$, the time $T'$ is uniform with respect to all operators $A$ and functions $f$ satisfying \ref{H_stab} to \ref{H_f_Lip}. We insist on the fact that $T'$ depends on the Banach spaces $X$ and $Z$: if Assumptions \ref{H_Z} to \ref{H_f_Lip} are satisfied for other Banach spaces $X$ and $Z$, then the time $T'$ may be different; denoting by $T'_{X,Z}$ this time, we can note that, if $X_1$ (resp., $Z_1$) is continuously embedded in $X_2$ (resp., $Z_2$), i.e., $X_1\hookrightarrow X_2$ and $Z_1\hookrightarrow Z_2$, and if Assumptions \ref{H_Z} to \ref{H_f_Lip} are satisfied for the two pairs of Banach spaces $(X_1,Z_1)$ and $(X_2,Z_2)$, then $T'_{X_1,Z_1}\leq T'_{X_2,Z_2}$.

It also follows from the proof that, taking $T'$ smaller if necessary, for every $\tilde y^0\in B_Z(y^0,r/2)$, there exists a unique solution $y(\cdot)\in\mathscr{C}^0([0,T'],Z)\cap\mathscr{C}^1([0,T'],X)$ of \eqref{abstract_quasilinear} such that $y(0)=\tilde y^0$, and taking its values in $B_Z(y^0,r)$ (similar statements can be found, e.g., in \cite{HughesKatoMarsden_ARMA1977}).
\end{remark}

\begin{proof}
The arguments can be found in \cite{Kato_1993, Sanekata_1989} (see also \cite{Kato_1975, Pazy} for a less general result but simpler proof), although not exactly in this form. We give a proof for completeness.

For every $r>0$,
for every $T'\in(0,T]$, let $\mathcal{S}_{T'}$ be the closed convex subset of all $z(\cdot)\in\mathscr{C}^0([0,T'],Z)$ such that $z(0)=y^0$ and $z(t)\in B_Z(y^0,r)$ for every $t\in[0,T']$.
Note that, by \ref{H_Z}, $\mathscr{C}^0([0,T'],Z) \subset \mathscr{C}^0([0,T'],X)$.
Given any $z(\cdot)\in\mathcal{S}_{T'}$, 
we consider the Cauchy problem
\begin{equation}\label{evol_eq_y_z}
\dot{y}(t) = A[t,z(t)]\, y(t) + f[t,z(t)] , \quad y(0)=y^0 \in Z .
\end{equation}
Using \ref{E_ode} and \eqref{evol_eq_y_z}, we obtain $\frac{d}{ds}(U_z(t,s)y(s))=U_z(t,s)f[s,y(s)]$, that we then integrate on $[0,t]$. Using \ref{H_f_Z} and \ref{E_UZ}, we conclude that \eqref{evol_eq_y_z} 
has a unique solution 
$y(\cdot)\in\mathscr{C}^0([0,T'],X)$, taking its values in $Z$, given by
\begin{equation}\label{sol_y_z}
y(t) = U_z(t,0)y^0 + \int_0^t U_z(t,s) f[s,z(s)]\, ds 
\end{equation}
for every $t\in[0,T']$, and we set $\Phi_{T'}(z(\cdot))=y(\cdot)$. This defines a map $\Phi_{T'}$ on $\mathcal{S}_{T'}$.

Let us prove that, actually, $y(\cdot)\in\mathscr{C}^0([0,T'],Z)\cap\mathscr{C}^1([0,T'],X)$.
Using \ref{H_S}, it suffices to prove that $Sy(\cdot)\in\mathscr{C}^0([0,T'],X)$.
We begin by noting that, using \ref{H_intertwining} and \ref{E_ode}, we have 
$SU_z(t,s) = U_z(t,s)S + \int_s^t U_z(t,\tau) B[\tau,z(\tau)] S U_z(\tau,s)\, d\tau$ 
on $Z$ (actually, this formula implies that $U_z(t,s)\in L(Z)$ depends continuously on $(t,s)$).
Then, using \eqref{sol_y_z}, we have 
\begin{multline}\label{sumoffourterms}
Sy(t) = U_z(t,0)Sy^0 + \int_0^t U_z(t,s) S f[s,z(s)] \, ds 
+ \int_0^t U_z(t,\tau) B[\tau,z(\tau)] S U_z(\tau,s) y^0 \, d\tau \\
+ \int_0^t \int_s^t U_z(t,\tau) B[\tau,z(\tau)] S U_z(\tau,s) f[s,z(s)] \, d\tau \, ds
\end{multline}
for every $t\in[0,T']$, 
and each of these four terms
is an element of $\mathscr{C}^0([0,T'],X)$ thanks to the various assumptions. The claim follows. 
In passing, note that, using the Fubini theorem in the fourth term at the right-hand side of \eqref{sumoffourterms}, the sum of the third and fourth terms is then equal to $\int_0^t U_z(t,\tau) B[\tau,z(\tau)] S y(\tau) \, d\tau$, and thus we get $Sy(\cdot) = \Psi_{T'}(z(\cdot),Sy(\cdot))$ where 
\begin{equation}\label{defPsiT}
\psi_{T'}(z(\cdot),x(\cdot))(t) = U_z(t,0)Sy^0 + \int_0^t U_z(t,s) S f[s,z(s)] \, ds 
+ \int_0^t U_z(t,\tau) B[\tau,z(\tau)] x(\tau) \, d\tau .
\end{equation}
This remark will be useful at the end of the proof.

Since $y(\cdot)=\Phi_{T'}(z(\cdot))\in\mathscr{C}^0([0,T'],Z)$, it follows that $\Phi_{T'}$ maps $\mathcal{S}_{T'}$ to $\mathcal{S}_{T'}$ if $T'$ is small enough.

Let us prove that $\Phi_{T'}$ is a contraction in $\mathscr{C}^0([0,T'],X)$ if $T'$ is small enough.
Since $U_{z_1}(t,s)-U_{z_2}(t,s) = -\int_s^t \frac{d}{dr} ( U_{z_1}(t,\tau) U_{z_2}(\tau,s) ) \, d\tau = \int_s^t U_{z_1}(t,\tau) ( A[\tau,z_1(\tau)] - A[\tau,z_2(\tau)] ) U_{z_2}(\tau,s) \, d\tau$, we infer from \ref{H_regZ}, \ref{E_estim_exp} and \ref{E_ode} that
\begin{equation}\label{Uz1z2}
\Vert (U_{z_1}(t,s)-U_{z_2}(t,s))\Vert_{L(Z,X)}
\leq C_4 M^2 T' e^{\vert\omega\vert T'} \Vert z_1(\cdot)-z_2(\cdot)\Vert_{\mathscr{C}^0([0,T'],X)} 
\end{equation}
for all $0\leq s\leq t\leq T'\leq T$ and all $z_1(\cdot),z_2(\cdot)\in\mathcal{S}_{T'}$. 
Applying \eqref{sol_y_z} to $y_1(\cdot)=\Phi_{T'}(z_1(\cdot))$ and $y_2(\cdot)=\Phi_{T'}(z_2(\cdot))$, we infer from \ref{H_f_Z}, \ref{H_f_Lip}, \ref{E_estim_exp} and \eqref{Uz1z2} that
$$
\Vert\Phi_{T'}(z_1(\cdot))-\Phi_{T'}(z_2(\cdot))\Vert_{\mathscr{C}^0([0,T'],X)} \leq T' C \Vert z_1(\cdot)-z_2(\cdot)\Vert_{\mathscr{C}^0([0,T'],X)} .
$$
for all $T'\in(0,T]$ and $z_1(\cdot),z_2(\cdot)\in\mathcal{S}_{T'}$, 
with $C= M e^{\vert\omega\vert T} ( C_7 + M C_4 + M T C_4 C_6 )$.
The contraction property follows by choosing $T'$ small enough.

In particular, $\Phi_{T'}$ is continuous in $\mathscr{C}^0([0,T'],X)$ norm. Hence, denoting by $\overline{\mathcal{S}}_{T'}$ the closure of $\mathcal{S}_{T'}$ in $\mathscr{C}^0([0,T'],X)$, $\Phi_{T'}$ maps the closed convex set $\overline{\mathcal{S}}_{T'}$ to itself and is a contraction, therefore it has a fixed point $y(\cdot)\in\overline{\mathcal{S}}_{T'}$ (in particular, $y(\cdot)\in\mathscr{C}^0([0,T'],X)$). 

It remains to prove that, actually, $y(\cdot)\in\mathcal{S}_{T'}$ (notably, $y(\cdot)\in\mathscr{C}^0([0,T'],Z)$). 
Note that $\overline{\mathcal{S}}_{T'} = \mathcal{S}_{T'}$ when $X$ and $Z$ are reflexive, so the following argument (developed in \cite{Kato_1993}) is only required in the absence of reflexivity.
Defining $y_0(\cdot)\in \mathscr{C}^0([0,T'],X)$ by $y_0(t)=y^0$ for any $t$, the fixed point $y(\cdot)$ is obtained as the limit in $\mathscr{C}^0([0,T'],X)$ of the sequence $(y_k(\cdot))_{k\in\N}$ of $\overline{\mathcal{S}}_{T'}$ defined by iteration $y_{k+1}(\cdot) = \Phi_{T'}(y_k(\cdot))$. 
Using the map $\Psi_{T'}$ defined by \eqref{defPsiT}, we therefore have $Sy_k(\cdot) = \Psi_{T'}(y_{k-1}(\cdot),Sy_k(\cdot))$ for every $k\in\N^*$.
It follows from \ref{H_intertwining} and \eqref{Uz1z2} that $\Psi_{T'}$ maps continuously $\overline{\mathcal{S}}_{T'}\times E$ to $E$, where $E$ is a closed ball of $\mathscr{C}^0([0,T'],X)$, of center $0$ and of sufficiently large radius, and moreover $\Psi_{T'}$ is contracting with respect to $x(\cdot)$ if $T'$ is chosen small enough. Let us prove that $(Sy_k(\cdot))_{k\in\N}$ is a Cauchy sequence in $\mathscr{C}^0([0,T'],X)$: this is then enough to conclude because it implies that $(y_k(\cdot))_{k\in\N}$ is a Cauchy sequence in $\mathscr{C}^0([0,T'],Z)$, hence it converges and the limit must be $y(\cdot)$.
Let $w(\cdot)\in E$ be such that $w(\cdot) = \Psi_{T'}(y(\cdot),w(\cdot))$ (it exists by the Banach fixed-point theorem). By the triangular inequality, we have
\begin{multline*}
\Vert Sy_k(\cdot)-w(\cdot)\Vert_{\mathscr{C}^0([0,T'],X)}
\leq \Vert \Psi_{T'}(y_{k-1}(\cdot),Sy_k(\cdot)) - \Psi_{T'}(y_{k-1}(\cdot),w(\cdot))\Vert_{\mathscr{C}^0([0,T'],X)} \\
+ \Vert \Psi_{T'}(y_{k-1}(\cdot),w(\cdot)) - \Psi_{T'}(y(\cdot),w(\cdot))\Vert_{\mathscr{C}^0([0,T'],X)} .
\end{multline*}
The first term at the right-hand side is less than $C\Vert Sy_k(\cdot)-w(\cdot)\Vert_{\mathscr{C}^0([0,T'],X)}$ for some $C>0$ because $\Psi_{T'}$ is contracting, and the second term converges to $0$ as $k\rightarrow+\infty$ by continuity of $\Psi_{T'}$. It follows that $Sy_k(\cdot)$ converges to $w(\cdot)$, which finishes the proof.
\end{proof}

\begin{remark}\label{rem_explicit}
When $A[t,y]=A$ does not depend on $(t,y)$ and generates a $C_0$ semigroup $(e^{tA})_{t\geq 0}$, we are in the very classical framework of semigroup theory (see \cite{EngelNagel,Pazy}). Assumptions \ref{H_Z} to \ref{H_intertwining} are satisfied with $Z=D(A)$, $S=A$, $B=0$, and we have $U(t,s)=e^{(t-s)A}$.

When $A[t,y]=A(t)$ does not depend on $y$, but depends on $t$, we are in the framework of linear evolution equations, treated for example in \cite[Chapter 5]{Pazy}. 

In these two cases where the operator does not depend on $y$, the Duhamel formula \eqref{duhamel_implicit} is explicit, because $U_y(t,s)$ does not depend on $y$.

\smallskip
We speak of a \emph{quasilinear} evolution equation when $A[t,y]$ depends on $y$. Then, $U_y(t,s)$ depends on $y(\cdot)$ and the Duhamel formula \eqref{duhamel_implicit} is implicit with respect to $y$.

The quasilinear theory for evolution systems has been considered and developed by many authors. Here, we have followed the presentation and assumptions done in \cite{Kato_1975, Kato_1993, Sanekata_1989} (see also \cite[Section 6.4]{Pazy}).
The framework covers, among many others, the following equations: 
Burgers, Korteweg-de Vries, 
hyperbolic systems of quasilinear partial differential equations of the first order,
Euler and Navier Stokes (incompressible) in $\R^3$, coupled Maxwell-Dirac, quasilinear waves, magnetohydrodynamics (including compressible fluids). 

An even more general theory exists, initiated by Kato in \cite{Kato_1967} with the notion of nonlinear semigroup, developed in the 70s with the famous Crandall-Liggett generation theorem (see \cite{CrandallLiggett}) or within maximal monotone operators (see \cite{Brezis_1973}). We refer to the unpublished book manuscript \cite{Benilan} and to the textbook \cite{ItoKappel} for a complete theory. 
For example, nonlinear semigroup theory covers the porous medium equation $\partial_t y = \triangle\varphi(u)$ and Hamilton-Jacobi equations, which are not covered by the quasilinear evolution equations theory.

Our choice of staying anyway in the quasilinear framework is motivated by its simplicity, by the fact that it already covers most of usual PDEs, and more technically, by the fact that the (implicit) Duhamel formula \eqref{duhamel_implicit} will be instrumental in deriving our main result, Theorem \ref{thm_approx_abstract}, in Section \ref{sec_abstract_main_result}.
\end{remark}

\subsection{A general approximation result}\label{sec_abstract_approx}
We assume that, for all $(t,z)\in[0,T]\times B_Z(y^0,r)$: 
\begin{enumerate}[label=$(H_{\arabic*})$] 
\setcounter{enumi}{7}
\item\label{H_Aepsilon} 
There exists a family of linear operators $A_\varepsilon[t,z]$ on $X$, indexed by $\varepsilon\in(0,1]$, satisfying Assumptions \ref{H_stab} to \ref{H_intertwining} (with $A$ replaced by $A_\varepsilon$) uniformly with respect to $\varepsilon$.
\item\label{H_Aepsilon_CV} There exists a Banach subspace $\hat Z$ of $Z$, dense in $Z$ and continuously embedded in $Z$, i.e., there exists $C_{\hat Z}>0$ such that $\Vert z\Vert_Z\leq C_{\hat Z}\Vert z\Vert_{\hat Z}$ for every $z\in\hat Z$,
and there exist $C_A>0$ and a continuous function $\chi_A:[0,1]\rightarrow[0,+\infty)$, satisfying $\chi_A(0)=0$, such that
$$
\Vert A_\varepsilon[t,z]-A[t,z]\Vert_{L(\hat Z,X)}\leq C_A \chi_A(\varepsilon)  \qquad\forall t\in[0,T] \qquad \forall z\in B_Z(y^0,r) \qquad\forall\varepsilon\in(0,1] .
$$
%
\end{enumerate}

We do not assume, for the moment, that $A_\varepsilon[t,z]$ is bounded, but in the next subsection we will focus on bounded approximations.
An example of bounded approximation operator $A_\varepsilon[t,z]$ satisfying Assumption \ref{H_Aepsilon}, not explicit but fully general, is given by the \emph{Yosida approximant} 
$$
A_\varepsilon[t,z] = J_\varepsilon[t,z] A[t,z] \qquad\textrm{where}\qquad J_\varepsilon[t,z] = \left( \mathrm{id}-\varepsilon A[t,z] \right)^{-1} ,
$$
which indeed satisfies, notably, the uniform stability estimate (see \cite{EngelNagel, Pazy}). 
Assumption \ref{H_Aepsilon_CV} refers to convergence estimates, which are often proved by explicit approximation constructions (see also \cite{ItoKappel_MC1998} for finite-dimensional approximations with error estimates), as we will do hereafter. 
For instance when $X=L^2(\Omega,\R^d)$, the Banach space $\hat Z$ may be a subspace of functions of $X$ having a certain number of bounded derivatives.

All in all, assuming \ref{H_stab} for $A_\varepsilon[t,z]$, uniformly with respect to $\varepsilon$, is the most stringent hypothesis. It is however classical in the Trotter-Kato theorem. It is usually established in practice by means of dissipativity properties, and this is also what we will do in the explicit construction hereafter. 

\medskip

Finally, we also assume that: 
\begin{enumerate}[label=$(H_{\arabic*})$, resume] 
\item\label{H_fepsilon} There exists a family of functions $f_\varepsilon:[0,T]\times X\rightarrow X$, indexed by $\varepsilon\in(0,1]$, satisfying Assumptions \ref{H_f_Z} and \ref{H_f_Lip} (with $f$ replaced by $f_\varepsilon$) uniformly with respect to $\varepsilon$.
\item\label{H_fepsilon_CV} There exist $C_f>0$ and a continuous function $\chi_f:[0,1]\rightarrow[0,+\infty)$, satisfying $\chi_f(0)=0$, such that
$$
\Vert f_\varepsilon[t,z]-f[t,z]\Vert_X\leq C_f \chi_f(\varepsilon)  \qquad\forall t\in[0,T] \qquad \forall z\in B_Z(y^0,r) \qquad\forall\varepsilon\in(0,1] .
$$
\end{enumerate}

For every $\varepsilon\in(0,1]$, we consider the quasilinear evolution equation
\begin{equation}\label{abstract_quasilinear_eps}
\boxed{
\dot y_\varepsilon(t) = A_\varepsilon[t,y_\varepsilon(t)] \, y_\varepsilon(t) + f_\varepsilon[t,y_\varepsilon(t)]
}
\end{equation}

\begin{proposition}\label{prop_yepsilon}
Under \ref{H_Z} to \ref{H_fepsilon}:
\begin{itemize}
\item For every $\varepsilon\in(0,1]$, there exists a unique solution $y_\varepsilon(\cdot)\in\mathscr{C}^0([0,T'],Z)\cap\mathscr{C}^1([0,T'],X)$ of \eqref{abstract_quasilinear_eps} such that $y_\varepsilon(0)=y^0$, where $T'\in(0,T]$ is the same as for the solution $y(\cdot)$ considered in Section \ref{sec_abstract_setting} and does not depend on $\varepsilon$. 
\item If $y\in L^\infty([0,T'],\hat Z)$ then
\begin{equation}\label{CV_yeps_abstract}
\begin{split}
\Vert y_\varepsilon(t)-y(t)\Vert_X
& \leq M ( C_A \Vert y\Vert_{L^\infty([0,T'],\hat Z)} \chi_A(\varepsilon) + C_f \chi_f(\varepsilon) ) \int_0^t e^{(\omega + M(C_4 r + C_7))s}\, ds \\
&\leq \mathrm{Cst}(\chi_A(\varepsilon) + C_f \chi_f(\varepsilon))
\end{split}
\end{equation}
for all $t\in[0,T']$ and $\varepsilon\in(0,1]$.
\end{itemize}
\end{proposition}

Of course, when $\hat Z=Z$, the assumption that $y\in L^\infty([0,T'],\hat Z)$ is satisfied and we have $\Vert y\Vert_{L^\infty([0,T'],\hat Z)} \leq r$.

Proposition \ref{prop_yepsilon} is similar to \cite[Theorem 7]{Kato_1975} and \cite[Theorem III]{Kato_1993}, where the above convergence result is proved for $\hat Z=Z$ without convergence estimate. 
Proposition \ref{prop_yepsilon} can thus be seen as a slight improvement, as we quantify the convergence. This is also why we have added the freedom of a space $\hat Z\subset Z$, anticipating that stronger convergence estimates can be obtained if the solution shares more regularity properties.

Note that Assumptions \ref{H_Aepsilon_CV} and \ref{H_fepsilon_CV} are not required in the first item of Proposition \ref{prop_yepsilon}.

Finally, recalling Remark \ref{rem_prop_existence_uniqueness}, we underline that the time $T'$ depends, in particular, on the spaces $X$ and $Z$. 

\begin{proof}
The existence and uniqueness on the whole interval $[0,T']$ comes from Proposition \ref{prop_existence_uniqueness} and Remark \ref{rem_prop_existence_uniqueness}.
Given any $\varepsilon\in(0,1]$, writing that 
$\frac{d}{dt}(y_\varepsilon(t) - y(t)) 
= A_\varepsilon[t,y_\varepsilon(t)]\, (y_\varepsilon(t)-y(t)) + (A_\varepsilon[t,y_\varepsilon(t)]-A[t,y(t)])\, y(t) + f_\varepsilon[t,y_\varepsilon(t)]-f[t,y(t)]$ and integrating, we infer from the Duhamel formula \eqref{duhamel_implicit} that, for every $t\in[0,T]$,
\begin{multline*}
y_\varepsilon(t)-y(t) = \int_0^t U_{y_\varepsilon}(t,s)\, \Big( \big( A_\varepsilon[s,y_\varepsilon(s)] - A_\varepsilon[s,y(s)] \big)\, y(s) \\ + \big( A_\varepsilon[s,y(s)] - A[s,y(s)] \big)\, y(s) 
+ f_\varepsilon[t,y_\varepsilon(t)]-f[t,y(t)] \Big) \, ds .
\end{multline*}
Noting that $y_\varepsilon(s), y(s)\in B_Z(y^0,r)$ by Proposition \ref{prop_existence_uniqueness}, using the (uniform) stability estimates \ref{E_estim_exp} for $U_{y_\varepsilon}$, the (uniform) Lipschitz properties \ref{H_regZ} for $A_\varepsilon$ and \ref{H_f_Lip} for $f_\varepsilon$, the convergence estimates \ref{H_Aepsilon_CV} and \ref{H_fepsilon_CV}, and the assumption that $y\in L^\infty([0,T'],\hat Z)$, it follows that
$$
\Vert y_\varepsilon(t)-y(t)\Vert_X
\leq M \int_0^t e^{\omega(t-s)} \big( (C_4 r + C_7) \Vert y_\varepsilon(s)-y(s)\Vert_X + C_A \Vert y\Vert_{L^\infty([0,T'],\hat Z)} \chi_A(\varepsilon) + C_f \chi_f(\varepsilon) \big) \, ds
$$
and therefore, by the Gronwall lemma applied to $e^{-\omega t}\Vert y_\varepsilon(t)-y(t)\Vert_X$, we obtain finally \eqref{CV_yeps_abstract}.
\end{proof}

\subsection{Finite particle approximation system}\label{sec_abstract_particle_approx}
In order to design a finite particle approximation system, we do a last assumption.
Recall that, at the beginning of Section \ref{sec_abstract}, we have assumed that the Banach space $X=\mathcal{F}(\Omega,\R^d)$ is a space of functions on $\Omega$.
This is now important, and in what follows we are going to use the Banach space $L^\infty(\Omega,\R^d)$, sometimes denoted $L^\infty$ for short.
%
Recall that $(\Omega,\mathrm{d}_\Omega)$ is a complete metric space, endowed with a probability measure  $\nu\in\mathcal{P}(\Omega)$, and having a family of tagged partitions.

In what follows, the set $\R^{d\times d}$ of square real-valued matrices of size $d$ is equipped with the matrix norm $\Vert\ \Vert_{\R^{d\times d}}$ induced by the Euclidean norm $\Vert\ \Vert_{\R^d}$ on $\R^d$. 

Keeping the same $r$ as in the previous assumptions, now we moreover assume that:
\begin{enumerate}[label=$(H_{\arabic*})$, resume] 
\item\label{H_sigmaepsilon} For every $\varepsilon\in(0,1]$, for all $(t,z)\in[0,T]\times B_{L^\infty}(y^0,r)$, the operator $A_\varepsilon[t,z]$ is well defined and bounded on $L^\infty(\Omega,\R^d)$, i.e., $A_\varepsilon[t,z]\in L(L^\infty)$, and can be written as
$$
A_\varepsilon[t,z]\, y(x) = \int_\Omega \sigma_\varepsilon[t,z](x,x')\, y(x')\, d\nu(x')
\qquad\forall y\in L^\infty(\Omega,\R^d) 
\qquad \forall x\in\Omega 
$$
where the kernel $\sigma_\varepsilon[t,z](x,x')\in\R^{d\times d}$ depends continuously on $(t,z,x,x')\in [0,T]\times B_{L^\infty}(y^0,r)\times\Omega\times\Omega$, is bounded and Lipschitz continuous with respect to $(z,x,x')$, uniformly with respect to $t$, i.e.,
\begin{equation*}
\begin{split}
\Vert\sigma_\varepsilon[t,z](x,x'))\Vert_{\R^{d\times d}} & \leq \Vert\sigma_\varepsilon\Vert_\infty ,  \\
 \Vert\sigma_\varepsilon[t,z_1](x_1,x'_1)-\sigma_\varepsilon[t,z_2](x_2,x'_2)\Vert_{\R^{d\times d}} 
 &\leq \Lip(\sigma_\varepsilon)\left( \Vert z_1-z_2\Vert_{L^\infty} + \mathrm{d}_\Omega(x_1,x_2) + \mathrm{d}_\Omega(x'_1,x'_2) \right) ,
\end{split}
\end{equation*}
for all $t\in[0,T]$, $z_1,z_2\in B_{L^\infty}(y^0,r)$ and $x,x',x_1,x_2,x'_1,x'_2\in\Omega$.
\item\label{H_fepsilon_Lip} $f_\varepsilon[t,z](x)$ depends continuously on $(t,z,x)\in [0,T]\times B_{L^\infty}(y^0,r)\times\Omega$, is bounded and Lipschitz continuous with respect to $(z,x)$, uniformly with respect to $t$, i.e.,
\begin{equation*}
\begin{split}
\Vert f_\varepsilon[t,z](x)\Vert_{\R^d} &\leq \Vert f_\varepsilon\Vert_\infty  ,  \\
\Vert f_\varepsilon[t,z_1](x_1)-f_\varepsilon[t,z_2](x_2)\Vert_{\R^d} 
&\leq \Lip(f_\varepsilon)\left( \Vert z_1-z_2\Vert_{L^\infty} + \mathrm{d}_\Omega(x_1,x_2) \right) ,
\end{split}
\end{equation*}
for all $t\in[0,T]$, $z_1,z_2\in B_{L^\infty}(y^0,r)$ and $x,x_1,x_2,x'_1,x'_2\in\Omega$.
\end{enumerate}

Assumption \ref{H_sigmaepsilon} is related to the regularity of the Schwartz kernel of $A_\varepsilon[t,z]$. For instance if $A[t,z]$ is a differential operator then \ref{H_sigmaepsilon} is satisfied by iterating the Yosida approximation, taking $A_\varepsilon[t,z]=J_\varepsilon^j A[t,z]$ for $j$ large enough. Of course, if $A[t,z]$ is unbounded then $\Vert\sigma_\varepsilon\Vert_\infty\rightarrow+\infty$ and $\Lip(\sigma_\varepsilon) \rightarrow+\infty$ as $\varepsilon\rightarrow 0$.

\medskip

Now, given any $\varepsilon\in(0,1]$, let us introduce the particle approximation of \eqref{abstract_quasilinear_eps}.
We use the family $(\mathcal{A}^N,x^N)_{N\in\N^*}$ of tagged partitions of $\Omega$ associated with $\nu$, satisfying \eqref{def_tagged}, with $\mathcal{A}^N=(\Omega_1^N,\ldots,\Omega_N^N)$ and $x^N=(x_1^N,\ldots,x_N^N)$.
Given any $N\in\N^*$, we consider the finite particle system 
\begin{equation}\label{particle_system_eps}
\boxed{
\dot\xi_{\varepsilon,i}^N(t) = \frac{1}{N}\sum_{j=1}^N \sigma_\varepsilon [t, y_\varepsilon^N(t)] (x_i^N,x_j^N) \, \xi_{\varepsilon,j}^N(t) + f_\varepsilon [t,y_\varepsilon^N(t)] (x_i^N)
\qquad \forall i\in\{1,\ldots,N\} 
}
\end{equation}
where $y_\varepsilon^N(t)\in L^\infty(\Omega,\R^d)$ is the piecewise function on $\Omega$ defined by
\begin{equation}\label{defyepsN}
\boxed{
y_\varepsilon^N(t,x) = \sum_{i=1}^N \xi_{\varepsilon,i}^N(t) \, \mathds{1}_{\Omega_i^N}(x) 
}
\end{equation}
where $\mathds{1}_{\Omega_i^N}$ is the characteristic function of $\Omega_i^N$, defined by $\mathds{1}_{\Omega_i^N}(x)=1$ if $x\in\Omega_i^N$ and $0$ otherwise. 
As said for $y$ at the beginning of Section \ref{sec_abstract}, we denote indifferently $y_\varepsilon^N(t)(x)=y_\varepsilon^N(t,x)$.
In particular, we have $y_\varepsilon^N(t,x_i^N)=\xi_{\varepsilon,i}^N(t)$ for every $i\in\{1,\ldots,N\}$. 

\begin{proposition}\label{prop_yepsilonN}
Assume that $y^0\in\mathscr{C}^0(\Omega,\R^d)$. For all $\varepsilon\in(0,1]$ and $N\in\N^*$, there exists a unique solution $t\mapsto\Xi_\varepsilon^N(t)=(\xi_{\varepsilon,1}^N(t),\ldots,\xi_{\varepsilon,N}^N(t))$ of \eqref{particle_system_eps} such that $\xi^N_i(0)=y^0(x^N_i)$ for every $i\in\{1,\ldots,N\}$, which is well defined on $[0,T']$ if $N$ is large enough. 
\end{proposition}

\begin{proof}
We are in the framework of Appendix \ref{app_abstract_quasilinear_integral}. Indeed, thanks to \ref{H_sigmaepsilon} and \ref{H_fepsilon_Lip}, for any fixed $\varepsilon$, the approximation evolution equation \eqref{abstract_quasilinear_eps} is a quasilinear integral evolution equation of the form \eqref{abstract_quasilinear_integral} in Appendix \ref{app_abstract_quasilinear_integral_setting}, and the finite particle system \eqref{particle_system_eps} is of the form \eqref{particle_system} in Appendix \ref{app_prop_existence_yN}.

There is however an important issue to be addressed. As already underline, the time $T'$ in Propositions \ref{prop_existence_uniqueness} and \ref{prop_yepsilon} depends, in particular, on the Banach spaces in which the equations are settled. We have first applied these propositions in the Banach spaces $X$ and $Z$, and obtained a time $T'$ of existence, that we denote by $T'_{X,Z}$ as in Remark \ref{rem_prop_existence_uniqueness}, for the solution $y$ of \eqref{abstract_quasilinear} and for the solution $y_\varepsilon$ of \eqref{abstract_quasilinear_eps}, and $T'_{X,Z}$ does not depend on $\varepsilon$.
Second, following Appendix \ref{app_abstract_quasilinear_integral}, for any $\varepsilon$ fixed, we apply Proposition \ref{prop_existence_uniqueness} to the approximation evolution equation \eqref{abstract_quasilinear_eps} in the Banach space $L^\infty(\Omega,\R^d)$ (i.e., with $X=Z=L^\infty$); we thus obtain another time of existence, denoted $T'_{L^\infty,L^\infty}$.

\end{proof}

Then, Proposition \ref{prop_yepsilonN} follows from Proposition \ref{prop_existence_yN} in Appendix \ref{app_prop_existence_yN}, the nontrivial fact being the well-posedness on the whole time interval $[0,T']$ if $N$ is large enough. 
Note that, in the classical semigroup case where $A[t,z]=A$, the finite particle system \eqref{particle_system_eps} is linear and solutions are defined on $\R$.


The particle system \eqref{particle_system_eps} is expected to provide a finite particle approximation of the quasilinear evolution equation \eqref{abstract_quasilinear}, in the sense that solutions $y$ of \eqref{abstract_quasilinear} are expected to be limits of $y_\varepsilon^N$ as $N\rightarrow+\infty$ and $\varepsilon\rightarrow 0$.
However, since the particle system \eqref{particle_system_eps} does not have any (classical) limit as $\varepsilon\rightarrow 0$, in order to derive convergence estimates we will have to let $N$ tend to $+\infty$ and $\varepsilon$ to $0$ at some appropriate scale. 
We therefore have to track the constants carefully in all estimates. This is what is done in Theorem \ref{thm_estim_graph} in Appendix \ref{app_thm_estim_graph}.

\subsection{Main result}\label{sec_abstract_main_result}
We work under Assumptions \ref{H_Z} to \ref{H_sigmaepsilon}, so that the results of Propositions \ref{prop_existence_uniqueness}, \ref{prop_yepsilon} and \ref{prop_yepsilonN} are available. 

\begin{theorem}\label{thm_approx_abstract}
Let $y^0\in Z\cap\Lip(\Omega,\R^d)$, and
let $y(\cdot)\in\mathscr{C}^0([0,T'],Z)\cap\mathscr{C}^1([0,T'],X)$ be the unique solution of \eqref{abstract_quasilinear} such that $y(0)=y^0$ (see Proposition \ref{prop_existence_uniqueness}). We assume that $y\in L^\infty([0,T'],\hat Z)$.

For any $\varepsilon\in(0,1]$ and any $N\in\N^*$ large enough, let $t\mapsto\Xi_\varepsilon^N(t)=(\xi_{\varepsilon,1}^N(t),\ldots,\xi_{\varepsilon,N}^N(t))$ be the unique solution of \eqref{particle_system_eps} on $[0,T']$ such that $\xi_{\varepsilon,i}^N(0)=y^0(x_i^N)$ for every $i\in\{1,\ldots,N\}$ (see Proposition \ref{prop_yepsilonN}), and let $y_\varepsilon^N$ be defined by \eqref{defyepsN}.
We set
\begin{equation*}
\begin{split}
a_1^\varepsilon &= C_\Omega\Lip(y^0) + \frac{C_\Omega}{a_3^\varepsilon} \left( 2\Lip(\sigma_\varepsilon)(r+\Vert y^0\Vert_{L^\infty}) + \Lip(f_\varepsilon) + \Vert\sigma_\varepsilon\Vert_\infty \Lip(y^0) \right),   \\
a_2^\varepsilon &= \frac{\Vert\sigma_\varepsilon\Vert_\infty}{a_3^\varepsilon} \left(  \Lip(\sigma_\varepsilon) r + \Lip(f) \right), \qquad
a_3^\varepsilon=\Lip(\sigma_\varepsilon)(r+\Vert y^0\Vert_{L^\infty}) + \Vert\sigma_\varepsilon\Vert_\infty + \Lip(f_\varepsilon) .
\end{split}
\end{equation*}
\begin{enumerate}[leftmargin=5mm,label=(\roman*)]
\item\label{abstracti} If there exists a continuous and dense embedding $L^\infty(\Omega,\R^d)\hookrightarrow X$, i.e., if there exists $C_\infty>0$ such that $\Vert z\Vert_X \leq C_\infty\Vert z\Vert_{L^\infty}$ for any $z\in L^\infty(\Omega,\R^d)$, then
\begin{multline}\label{thm_approx_abstract_estimate_1}
\Vert y_\varepsilon^N(t) - y(t) \Vert_X
\leq M ( C_A \Vert y\Vert_{L^\infty([0,T'],\hat Z)} \chi_A(\varepsilon) + C_f \chi_f(\varepsilon) ) \int_0^t e^{(\omega + M(C_4 r + C_7))s}\, ds \\
+ \frac{C_\infty}{N^\gamma} (a_1^\varepsilon+a_2^\varepsilon t) e^{a_3^\varepsilon t} 
\end{multline}
for all $t\in[0,T']$, $\varepsilon\in(0,1]$ and $N\in\N^*$ large enough.
\item\label{abstractii} If there exists a continuous and dense embedding $X\hookrightarrow L^\infty(\Omega,\R^d)$, i.e., if there exists $C_\infty>0$ such that $\Vert z\Vert_{L^\infty} \leq C_\infty\Vert z\Vert_X$ for any $z\in X$, then
\begin{multline}\label{thm_approx_abstract_estimate_2}
\Vert y_\varepsilon^N(t) - y(t) \Vert_{L^\infty}
\leq C_\infty M ( C_A \Vert y\Vert_{L^\infty([0,T'],\hat Z)} \chi_A(\varepsilon) + C_f \chi_f(\varepsilon) ) \int_0^t e^{(\omega + M(C_4 r + C_7))s}\, ds \\
+ \frac{1}{N^\gamma} (a_1^\varepsilon+a_2^\varepsilon t) e^{a_3^\varepsilon t} 
\end{multline}
for all $t\in[0,T']$, $\varepsilon\in(0,1]$ and $N\in\N^*$ large enough.
\end{enumerate}
\end{theorem}

\begin{proof}
%
%
The proof, which is easy, is done in three steps.

%
%

As a first step, we apply Proposition \ref{prop_yepsilon}, which establishes that $y_\varepsilon$ converges to $y$, with the convergence estimate \eqref{CV_yeps_abstract}.

As a second step, given any fixed $\varepsilon\in(0,1]$, we apply Theorem \ref{thm_estim_graph} of Appendix \ref{app_thm_estim_graph} to prove that $y_\varepsilon^N$ converges to $y_\varepsilon$, with the convergence estimate
\begin{equation}\label{estim_yepsilon_yNepsilon}
\Vert y_\varepsilon(t,x)-y_\varepsilon^N(t,x)\Vert_{\R^d} 
\leq \frac{1}{N^\gamma} (a_1^\varepsilon+a_2^\varepsilon t) e^{a_3^\varepsilon t} 
\end{equation}
for all $(t,x)\in[0,T']\times\Omega$ and $N\in\N^*$ large enough.

As a third step, we conclude by the triangular inequality, as follows.
In case \ref{abstracti}, we have $\Vert y_\varepsilon^N(t) - y_\varepsilon(t,\cdot) \Vert_{X} \leq C_\infty \Vert y_\varepsilon^N(t) - y_\varepsilon(t) \Vert_{L^\infty}$.
Using the triangular inequality, 
we infer \eqref{thm_approx_abstract_estimate_1} from \eqref{CV_yeps_abstract} and \eqref{estim_yepsilon_yNepsilon}.
The argument is similar in case \ref{abstractii}.
%
%
\end{proof}

\begin{remark}[Comments on Theorem \ref{thm_approx_abstract}.]\label{rem_comments_thm_approx_abstract}
To illustrate and understand the convergence estimates \eqref{thm_approx_abstract_estimate_1} and \eqref{thm_approx_abstract_estimate_2}, let us assume that $\chi_A(\varepsilon) + \chi_f(\varepsilon)\sim\varepsilon$ and that $\Vert\sigma_\varepsilon\Vert_\infty+\Lip(\sigma_\varepsilon)\sim\frac{1}{\varepsilon^k}$ for some $k\in\N^*$ (this will be the case in the explicit construction hereafter), as $\varepsilon\rightarrow 0$.
Then, ignoring constants, the right-hand side of \eqref{thm_approx_abstract_estimate_1} or \eqref{thm_approx_abstract_estimate_2} is of the order of 
$$
\varepsilon + \frac{1}{N^\gamma} \frac{e^{1/\varepsilon^k}}{\varepsilon^k} .
$$
In order to pass to the limit as $N\rightarrow+\infty$ and $\varepsilon\rightarrow 0$, it is appropriate to choose parameters such that this term tends to $0$.
An optimization argument shows that the best choice for $\varepsilon$ in function of $N$ is
$
\varepsilon_N \sim 1/(\ln N)^{1/k}
$
as $N\rightarrow+\infty$, and in this case the estimate \eqref{thm_approx_abstract_estimate_1}, applied, typically, with $X=L^2(\Omega,\R^d)$ gives
$$
\Vert y_{\varepsilon_N}^N(t) - y(t) \Vert_{L^2}
\leq \frac{\mathrm{Cst}}{(\ln N)^{1/k}}  .
$$
%
The above general estimate can certainly be improved under additional assumptions, for example, taking into account the physical context like energy conservation properties. Indeed, the estimate \eqref{estim_yepsilon_yNepsilon}, yielded by Theorem \ref{thm_estim_graph} of Appendix \ref{app_thm_estim_graph}, is obtained under general assumptions. 
\end{remark}

\begin{remark}
As a consequence of the first item of Theorem \ref{thm_estim_graph} of Appendix \ref{app_thm_estim_graph}, we could relax in Theorem \ref{thm_approx_abstract} the assumption that $y^0\in\Lip(\Omega,\R^d)$ to $y^0\in\mathscr{C}^0(\Omega,\R^d)$, but in this case we would get weaker estimates \eqref{thm_approx_abstract_estimate_1} and \eqref{thm_approx_abstract_estimate_2}: the term $\frac{1}{N^\gamma}$ is replaced by a $\mathrm{o}(1)$ as $N\rightarrow+\infty$, and we cannot do the diagonal convergence argument of Remark \ref{rem_comments_thm_approx_abstract}. This is why tracking convergence estimates is crucial in our analysis. 
\end{remark}

\section{Application to quasilinear PDEs}\label{sec_PDE}
\subsection{A general class of quasilinear PDEs}
Let us first define a domain $\Omega$, equipped with a metric $\mathrm{d}_\Omega$ and with a probability measure $\nu$, for setting a PDE on it, in view, then, of considering particles evolving in $\Omega$, approximating the solutions of this evolution equation. Let $n\in\N^*$. In the sequel we assume either that:
\begin{enumerate}[label=$(O_\arabic*)$, topsep=2mm]
\item\label{O1} $\Omega$ is the compact closure of a bounded open subset of $\R^n$ with a Lipschitz boundary, $\mathrm{d}_\Omega$ is the induced Euclidean distance, $\nu$ is the Lebesgue measure on $\Omega$ and (without loss of generality) the Lebesgue volume of $\Omega$ is equal to $1$;
\end{enumerate}
or that:
\begin{enumerate}[label=$(O_\arabic*)$,resume, topsep=2mm]
\item\label{O2} $\Omega$ is a smooth compact Riemannian manifold of dimension $n$ (without boundary), $\mathrm{d}_\Omega$ is its Riemannian distance, and $\nu$ is the canonical Riemannian probability measure.
\end{enumerate}
In the case \ref{O1}, $\Omega$ is usually called a Lipschitz compact domain of $\R^n$. 
In the case \ref{O2}, for example $\Omega$ may be the sphere or the torus of dimension $n$.
In what follows, in both cases, in spite of a slight ambiguity in case \ref{O2}, to simplify the notations an integral $\int_\Omega f\, d\nu$ will be denoted $\int_\Omega f(x)\, dx$; it is thus understood that $\int_\Omega 1\, dx = 1$.

Under \ref{O1} or \ref{O2}, there always exist families of tagged partitions associated with $\nu$ satisfying \eqref{def_tagged} with $\gamma=1/n$. 


\medskip

In local coordinates $x$ on $\Omega$, we denote $D^\alpha = \partial_1^{\alpha_1}\cdots\partial_n^{\alpha_n}$ where $\partial_i$ is the partial derivative with respect to the $i^\textrm{th}$ variable of $x$ (which we do not denote by $x_i$ because the notation is already used for the tagged partitions), where $\alpha=(\alpha_1,\ldots,\alpha_n)\in\N^n$ and we set $\vert\alpha\vert=\sum_{i=1}^n\alpha_i$.

Let $p\in\N^*$ and $T>0$ be arbitrary. 
Throughout the section, 
we assume that $X=L^2(\Omega,\R^d)$ 
and that $Z\subset H^s(\Omega,\R^d)$ for some $s\in\R$ large enough, to be chosen. 
We consider the quasilinear partial differential system
\begin{equation}\label{general_quasilinear_PDE}
\boxed{
\partial_t y(t,x) 
= \sum_{\vert\alpha\vert\leq p} a_\alpha[t,y(t)](x) \, D^\alpha y(t,x)+f[t,y(t)](x)
}
\end{equation}
where $y(t,x)\in\R^d$, with some prescribed conditions at the boundary of $\Omega$ when $\Omega$ has a boundary (they are involved in the definition of the domain of the operator), and with an initial condition $y(0)=y^0\in Z$. 
Here, for every $\alpha\in\N^k$ such that $\vert\alpha\vert\leq p$, 
for all $(t,z)\in[0,T]\times Z$, $a_\alpha[t,z]$ is a function on $\Omega$, taking its values in the set $\R^{d\times d}$ of real square matrices of size $d$. 
The system \eqref{general_quasilinear_PDE} is of the form \eqref{abstract_quasilinear} with 
$$
A[t,z] = \sum_{\vert\alpha\vert\leq p} a_\alpha[t,z](\cdot) \, D^\alpha  . 
$$
We assume that 
there exists $r>0$ such that, for all $(t,z)\in[0,T]\times B_Z(y^0,r)$:
\begin{itemize}
\item The operator $A[t,z]$ on $X=L^2(\Omega,\R^d)$ is defined on a domain $D(A[t,z])\subset H^p(\Omega,\R^d)$, dense in $L^2(\Omega,\R^d)$, which may encode some Dirichlet or Neumann like boundary conditions, maybe of higher order.
\item $s$ is large enough so that $Z\subset D(A[t,z])$ (a necessary condition is that $s\geq p$);
\item $s>d/2$, so that $Z\hookrightarrow L^\infty(\Omega,\R^d)$ by Sobolev embedding.
\item $Z$ is such that \ref{H_S} and \ref{H_intertwining} are satisfied.
\item For every $\alpha\in\N^n$ such that $\vert\alpha\vert\leq p$, $a_\alpha[t,z](x)\in\R^{d\times d}$ depends continuously on $(t,z,x)\in[0,T]\times B_Z(y^0,r)\times\Omega$ and is bounded and Lipschitz continuous with respect to $(z,x)$, uniformly with respect to $t$, i.e., there exists $C_a>0$ such that
\begin{equation*}
\begin{split}
\Vert a_\alpha[t,z](x)\Vert_{\R^{d\times d}} &\leq C_a , \\
\Vert a_\alpha[t,z_1](x_1)-a_\alpha[t,z_2](x_2)\Vert_{\R^{d\times d}} 
&\leq C_a \left( \Vert z_1-z_2\Vert_X + d_\Omega(x_1,x_2) \right) ,
\end{split}
\end{equation*}
for all $t\in[0,T]$, $z_1,z_2\in B_Z(y^0,r)$ and $x,x_1,x_2\in\Omega$.
\item $A[t,z]-\omega\,\mathrm{id}$ is $m$-dissipative, which means that it is dissipative (i.e., $\langle(A[t,z]-\omega\,\mathrm{id})f,f\rangle_{L^2}\leq 0$ for every $f\in D(A[t,z])$) and that $\mathrm{Ran}((\omega+1)\,\mathrm{id}-A[t,z]) = ((\omega+1)\,\mathrm{id}-A[t,z])D(A[t,z])=L^2(\Omega,\R^d)$.
Equivalently, by the Lumer-Phillips theorem (see \cite{EngelNagel}), $A[t,z]-\omega\,\mathrm{id}$ generates a $C_0$ semigroup of contractions in $X=L^2(\Omega,\R^d)$.
Equivalently, \ref{H_stab} is satisfied with $M=1$. 
\end{itemize}
Under the above assumptions, Assumptions \ref{H_Z} to \ref{H_f_Lip} are satisfied.

The Lipschitz property with respect to $x$ in the second requirement and the third requirement are not necessary to apply Proposition \ref{prop_existence_uniqueness} but they will be used in the explicit approximation procedure, later. 

By Proposition \ref{prop_existence_uniqueness}, there exists a unique solution $y(\cdot)\in\mathscr{C}^0([0,T'],Z)\cap\mathscr{C}^1([0,T'],X)$ of \eqref{general_quasilinear_PDE} such that $y(0)=y^0$, for some $T'\in(0,T]$. 
Moreover, $y(t)\in B_Z(y^0,r)$ for every $t\in[0,T']$.

As an application of Theorem \ref{thm_approx_abstract}, our objective is to prove that, under appropriate assumptions, the solution $y(\cdot)$ can be approximated by the solutions of a family of finite particle systems, for which we design an explicit construction.


\subsection{Main result}

\paragraph{Explicit particle approximation.}
Let $\eta\in \mathscr{C}^\infty_c(\R^n)$ be a nonnegative symmetric smooth real-valued function on $\R^n$, of compact support contained in the unit ball $B_{\R^n}(0,1)$, 
such that $\int_{\R^n}\eta(x)\, dx=1$. 
Here, symmetric means that $\eta(x)=\eta(-x)$ for every $x\in\R^n$.
We set $C_\eta = \int_{\R^n}\Vert x\Vert\eta(x)\, dx$.
For example, we can take
$$
\eta(x) = \left\{\begin{array}{ll} 
c\, e^{1/(\Vert x\Vert^2-1)} & \textrm{if}\ \Vert x\Vert < 1, \\
0 & \textrm{otherwise} ,
\end{array}\right.
$$
where $c>0$ is a normalization constant.
Given any $\varepsilon\in(0,1]$, we denote by $\eta_\varepsilon\in \mathscr{C}^\infty_c(\R^n)$  the (mollifier) function given by
$$
\eta_\varepsilon(x) = \frac{1}{\varepsilon^n}\eta\left(\frac{x}{\varepsilon}\right) \qquad\forall x\in\R^n .
$$
For all $(t,z)\in[0,T]\times Z$, we define $\sigma_\varepsilon[t,z](x,x')\in\R^{d\times d}$  by
\begin{equation}\label{def_sigma_eps}
\sigma_\varepsilon[t,z](x,x') 
=  \sum_{\vert\alpha\vert\leq p} \int_\Omega \eta_\varepsilon(x-x'') a_\alpha[t,z](x'')  (D^\alpha\eta_\varepsilon)(x''-x') \, d\nu(x'')
\qquad \forall x,x'\in\Omega\times\Omega . 
\end{equation}
We have $\sigma_\varepsilon[t,z]\in\mathscr{C}^\infty(\Omega\times\Omega,\R^{d\times d})$ (it is smooth up to the boundary) and in particular it is bounded above by $\Vert\sigma_\varepsilon\Vert_\infty$ and Lipschitz continuous, uniformly with respect to $(t,z)\in[0,T]\times B_Z(y^0,r)$, with
$$
\Vert\sigma_\varepsilon\Vert_\infty \leq \frac{C_\sigma}{\varepsilon^{n+p}} 
\qquad\textrm{and}\qquad
\Lip(\sigma_\varepsilon) \leq \frac{C_\sigma}{\varepsilon^{n+p+1}} 
$$
for some constant $C_\sigma>0$ depending on $\eta$ and on the bounds on $a_\alpha$ but not depending on $\varepsilon$.

\medskip
Let $(\mathcal{A}^N,x^N)_{N\in\N^*}$ be a family of tagged partitions of $\Omega$ associated with $\nu$ satisfying \eqref{def_tagged} with $\gamma=1/n$, with $\mathcal{A}^N=(\Omega_1^N,\ldots,\Omega_N^N)$ and $x^N=(x_1^N,\ldots,x_N^N)$.
Given any $\varepsilon\in(0,1]$ and any $N\in\N^*$, we consider the finite particle system \eqref{particle_system_eps}, with $\sigma_\varepsilon$ defined by \eqref{def_sigma_eps}, i.e.,
$$
\dot\xi_{\varepsilon,i}^N(t) = \frac{1}{N}\sum_{j=1}^N \sum_{\vert\alpha\vert\leq p} \int_\Omega \eta_\varepsilon(x_i^N-x'') a_\alpha[t,y_\varepsilon^N(t)](x'')  (D^\alpha\eta_\varepsilon)(x''-x_j^N) \, dx'' \, \xi_{\varepsilon,j}^N(t) + f [t,y_\varepsilon^N(t)] (x_i^N)
$$
for every $i\in\{1,\ldots,N\}$, where $y_\varepsilon(t) = \sum_{i=1}^N \xi_{\varepsilon,i}^N(t) \, \mathds{1}_{\Omega_i^N}$ (see \eqref{defyepsN}.
Recall that, by Proposition \ref{prop_yepsilonN}, if $y^0\in\Lip(\Omega,\R^d)$ then the particle system has a unique solution such that $\xi_{\varepsilon,i}^N(0)=y^0(x_i^N)$ for every $i\in\{1,\ldots,N\}$, well defined on on $[0,T']$ for $N$ large enough.

\begin{theorem}\label{thm_approx_PDE_particle}
%
Assume that $y^0\in Z\cap\Lip(\Omega,\R^d)$ 
and that $y\in L^\infty([0,T'], W^{p+1,\infty}(\Omega,\R^d))$.
%
There exists $C>0$ such that
\begin{equation}\label{thm_approx_PDE_particle_estimate}
\Vert y_\varepsilon^N(t) - y(t) \Vert_{L^2}
\leq C\left( \varepsilon + \frac{1}{N^{1/n}} \frac{e^{1/\varepsilon^{n+p+1}}}{\varepsilon^{n+p}} \right) 
\end{equation}
for all $t\in[0,T']$, $\varepsilon\in(0,1]$ and $N\in\N^*$ large enough,
except in case \ref{O1} when $n=1$, in which case the first term $\varepsilon$ in the parenthesis at the right-hand side of \eqref{thm_approx_PDE_particle_estimate} must be replaced with $\sqrt{\varepsilon}$.

As a consequence, taking $\varepsilon=\varepsilon_N = 1/(\ln N)^{1/(n+p+1)}$ as $N\rightarrow+\infty$ (see Remark \ref{rem_comments_thm_approx_abstract}), 
\eqref{thm_approx_PDE_particle_estimate} gives
\begin{equation*}
\boxed{
\Vert y_{\varepsilon_N}^N(t) - y(t) \Vert_{L^2}
\leq \frac{C}{(\ln N)^{1/(n+p+1)}}  
}
\end{equation*}
for all $t\in[0,T']$ and $N\in\N^*$ large enough. 
\end{theorem}

\begin{remark}
In the above explicit finite particle system, we have let untouched the function $f$ but we could of course approximate it as well, for example with a convolution.

A second remark is that we have used the functional space $X=L^2(\Omega,\R^d)$, but we can develop the same strategy for other spaces, like $X=\mathscr{C}^0(\Omega,\R^d)$. 
\end{remark}

\subsection{Proof of Theorem \ref{thm_approx_PDE_particle}}
\label{sec_proof_thm_approx_PDE_particle}
We are going to apply the item \ref{abstracti} of Theorem \ref{thm_approx_abstract} with 
the above spaces $X$ and $Z$, and with $\hat Z=W^{p+1,\infty}(\Omega,\R^d)$. The only thing we have to ensure is that Assumption \ref{H_Aepsilon_CV} is satisfied, with $\chi_A(\varepsilon) = \mathrm{O}(\sqrt{\varepsilon^2+\varepsilon^n})$ as $\varepsilon\rightarrow 0$.
This will be established in Lemma \ref{lem_CV_Aeps} at the end of this section.

Recall that the operator $A_\varepsilon[t,z]$ is defined by $A_\varepsilon[t,z] y(x) = \int_{\Omega} \sigma_\varepsilon[t,z](x,x') y(x')\, dx'$ for every $y\in Z$ and every $x\in\Omega$. 
To establish the required convergence estimate, we are going to express $A_\varepsilon[t,z]$ using an unusual convolution that we introduce next.
Since all what is developed hereafter is done for fixed $(t,z)\in[0,T]\times Z$, in the sequel we denote for short $A_\varepsilon=A_\varepsilon[t,z]$, $\sigma_\varepsilon=\sigma_\varepsilon[t,z]$ and $a_\alpha=a_\alpha[t,z]$.

\paragraph{Definition and properties of a convolution operator.}
Given any $g\in L^2(\Omega,\R^d)$, let us define and give some properties of the smooth approximation $\eta_\varepsilon\star_\Omega g\in \mathscr{C}^\infty(\Omega,\R^d)$ of $g$ for any $\varepsilon\in(0,1]$, defined hereafter. 

In the case \ref{O2}, i.e., when $\Omega$ is a smooth compact Riemannian manifold of dimension $n$ (without boundary), using a smooth partition of unity over an atlas of $\Omega$, we can always write $g=\sum_{i=1}^m g_i$ for some $m\in\N^*$ and for some functions $g_i\in L^k(\Omega,\R^d)$ whose essential support is contained in a chart of the atlas. In each chart, $\eta_\varepsilon\star g_i$ can thus be defined as the standard convolution in $\R^n$ for every $\varepsilon>0$ sufficiently small. At the global level, this defines the function $\eta_\varepsilon\star_\Omega g\in \mathscr{C}^\infty(\Omega)$.

In the case \ref{O1}, i.e., when $\Omega$ is the compact closure of a bounded open subset of $\R^n$ with a Lipschitz boundary, we have to be careful with the boundary. 
Given any $g\in L^2(\Omega,\R^d)$, for any $\varepsilon\in(0,1]$, we define the function $\eta_\varepsilon\star_\Omega g:\Omega\rightarrow\R^d$ by
$$
\eta_\varepsilon\star_\Omega g(x) = \int_\Omega \eta_\varepsilon(x-x') g(x')\, dx' \qquad\forall x\in\Omega
$$
but we stress that this is not a usual convolution (the integral is performed on $\Omega$ only) and thus the usual properties of the convolution cannot be used directly. This is why, hereafter, we relate this unusual convolution with the usual one, by extending functions on $\Omega$ to $\R^n$ by $0$ outside of $\Omega$. 
Given any $g\in L^2(\Omega,\R^d)$, 
we denote by $\tilde g = g\, \mathds{1}_\Omega\in L^2(\Omega,\R^d)$ the extension of $g$ to $\R^n$ by $0$.
For any $\varepsilon\in(0,1]$, we consider the function $\eta_\varepsilon\star\tilde g\in \mathscr{C}^\infty_c(\R^n,\R^d)$ defined by the usual convolution
$$
(\eta_\varepsilon\star\tilde g)(x) = \int_{\R^n} \eta_\varepsilon(x-x')\tilde g(x')\, dx' 
= \int_\Omega \eta_\varepsilon(x-x') g(x')\, dx' 
\qquad\forall x\in\R^n ,
$$
whose support satisfies $\supp(\eta_\varepsilon\star\tilde g)\subset\Omega+\overline{B}_{\R^n}(0,\varepsilon)$. We have
$$
\eta_\varepsilon\star_\Omega g = (\eta_\varepsilon\star\tilde g)_{\vert\Omega} ,
$$
i.e., $\eta_\varepsilon\star_\Omega g$ is the restriction of $\eta_\varepsilon\star\tilde g$ to $\Omega$.
Hence $\eta_\varepsilon\star_\Omega g\in \mathscr{C}^\infty(\Omega,\R^d)$: it is smooth up to the boundary of the compact domain $\Omega$. 
We also have $\eta_\varepsilon\star_\Omega g = (\tilde g\star\eta_\varepsilon)_{\vert\Omega}$.
Finally, as a consequence of the properties of the usual convolution, we have $\eta_\varepsilon\star_\Omega g\rightarrow g$ in $L^2(\Omega,\R^d)$ as $\varepsilon\rightarrow 0$.

More generally, for every $\alpha=(\alpha_1,\ldots,\alpha_n)\in\N^n$, the functions $D^\alpha(\eta_\varepsilon\star_\Omega g)$, $D^\alpha(\eta_\varepsilon)\star_\Omega g$ and $\eta_\varepsilon\star_\Omega D^\alpha g$ (provided that $D^\alpha g\in L^2(\Omega,\R^d)$ in the latter case) 
are smooth on $\Omega$ and are 
the restrictions to $\Omega$ of the smooth functions $D^\alpha(\eta_\varepsilon\star\tilde g)$, $D^\alpha(\eta_\varepsilon)\star \tilde g$ and $\eta_\varepsilon\star \widetilde{D^\alpha g}$ on $\R^n$, respectively. 
In particular, the function $A(\eta_\varepsilon\star_\Omega g)$ is the restriction to $\Omega$ of $A(\eta_\varepsilon\star\tilde g)$.

With these definitions, in both cases \ref{O1} and \ref{O2}, for every $\alpha=(\alpha_1,\ldots,\alpha_n)\in\N^n$ we have $D^\alpha (\eta_\varepsilon\star_\Omega g) = (D^\alpha \eta_\varepsilon)\star_\Omega g = \eta_\varepsilon\star_\Omega D^\alpha g$ for every $g\in L^2(\Omega,\R^d)$ (such that $D^\alpha g\in L^2(\Omega,\R^d)$ for the last equality) 
and $D^\alpha (\eta_\varepsilon\star_\Omega g)\rightarrow D^\alpha g$ in $L^2(\Omega,\R^d)$ as $\varepsilon\rightarrow 0$ if $D^\alpha g\in L^2(\Omega,\R^d)$.
%

Note that the $\R^{d\times d}$-valued function $\sigma_\varepsilon$ defined by \eqref{def_sigma_eps} is also given by
$$
\sigma_\varepsilon(x,x') = \Big( \eta_\varepsilon \star_\Omega \sum_{\vert\alpha\vert\leq p} a_\alpha(\cdot) (D^\alpha\eta_\varepsilon)(\cdot-x') \Big) (x) \qquad \forall x,x'\in\Omega.
$$

\paragraph{Expressing $A_\varepsilon$ with the convolution operator $\eta_\varepsilon\star_\Omega$.}
\begin{lemma}\label{lem_Aeps}
Given any $\varepsilon\in(0,1]$ and any $g\in L^2(\Omega,\R^d)$, we have\footnote{The function $A(\eta_\varepsilon\star_\Omega g)$ is the restriction of $A(\eta_\varepsilon\star \tilde g)$ to $\Omega$ and, denoting by $g_\varepsilon = A(\eta_\varepsilon\star \tilde g)\mathds{1}_\Omega$ the extension of the function $A(\eta_\varepsilon\star_\Omega g)$ to $\R^n$ by $0$, the function $\eta_\varepsilon\star_\Omega A(\eta_\varepsilon\star_\Omega g)$ is the restriction to $\Omega$ of the function $\eta_\varepsilon\star g_\varepsilon$.}
$$
A_\varepsilon g=\eta_\varepsilon\star_\Omega A(\eta_\varepsilon\star_\Omega g)
= \Big( \eta_\varepsilon\star \big( A(\eta_\varepsilon\star\tilde g)\mathds{1}_\Omega \big) \Big)_{\vert\Omega}  .   
$$
\end{lemma}

\begin{proof}
For $x\in\Omega$ fixed, we have, using that $D^\alpha\eta_\varepsilon\star_\Omega g = D^\alpha(\eta_\varepsilon\star_\Omega g)$,
$$
(A_\varepsilon g)(x) = \int_\Omega \sigma_\varepsilon(x,x')g(x')\, dx' = \Big( \eta_\varepsilon\star_\Omega \sum_{\vert\alpha\vert\leq p} a_\alpha D^\alpha\eta_\varepsilon\star_\Omega g \Big) (x)  = \Big( \eta_\varepsilon\star_\Omega A(\eta_\varepsilon\star_\Omega g) \Big) (x) ,
$$
thus giving the lemma.
\end{proof}

\paragraph{Uniform stability property.}
Thanks to Lemma \ref{lem_Aeps}, we can now establish that \ref{H_stab} is satisfied with $M=1$ (contraction semigroup) uniformly, with respect to $\varepsilon\in(0,1]$.

\begin{lemma}\label{convol_Omega_sym}
For all $g,h\in L^2(\Omega,\R^d)$, we have $\langle\eta_\varepsilon\star_\Omega g,h\rangle_{L^2} = \langle g,\eta_\varepsilon\star_\Omega h\rangle_{L^2}$.
\end{lemma}

\begin{proof}
Using that $\eta_\varepsilon\star_\Omega g = (\eta_\varepsilon\star\tilde g)_{\vert\Omega}$ and that $\tilde h=0$ on $\R^n\setminus\Omega$, we have $\langle \eta_\varepsilon\star_\Omega g,h\rangle_{L^2} = \langle \eta_\varepsilon\star\tilde g, \tilde h\rangle_{L^2(\R^n,\R^d)}$. Now, using the fact that $\eta_\varepsilon$ is symmetric, i.e., that $\eta_\varepsilon(x)=\eta_\varepsilon(-x)$ for any $x\in\R^n$, ensuring that the usual convolution by $\eta_\varepsilon$ is symmetric in $L^2(\R^n,\R^d)$, we have $\langle \eta_\varepsilon\star\tilde g, \tilde h\rangle_{L^2(\R^n,\R^d)} = \langle\tilde g, \eta_\varepsilon\star\tilde h\rangle_{L^2(\R^n,\R^d)}$. But the latter term is equal to $\langle g,\eta_\varepsilon\star_\Omega h\rangle_{L^2}$ because $\tilde g=0$ on $\R^n\setminus\Omega$. The lemma follows. 
\end{proof}

\begin{lemma}\label{lem_Aeps_dissip}
Like the operator $A-\omega\,\mathrm{id}$, the operator $A_\varepsilon-\omega\,\mathrm{id}$ is $m$-dissipative on $L^2(\Omega,\R^d)$, for any $\varepsilon\in(0,1]$. As a consequence, \ref{H_stab} is satisfied (with $M=1$), uniformly with respect to $\varepsilon\in(0,1]$.
\end{lemma}

\begin{proof}
Given any $g\in\mathscr{C}^\infty(\Omega,\R^d)$, applying Lemma \ref{convol_Omega_sym} to $g=(A-\omega\,\mathrm{id})(\eta_\varepsilon\star_\Omega g)$, we have
$$
\langle (A_\varepsilon-\omega\,\mathrm{id}) g,g\rangle_{L^2} 
= \langle \eta_\varepsilon\star_\Omega (A-\omega\,\mathrm{id})(\eta_\varepsilon\star_\Omega g),g\rangle_{L^2} 
= \langle (A-\omega\,\mathrm{id})(\eta_\varepsilon\star_\Omega g), \eta_\varepsilon\star_\Omega g\rangle_{L^2} 
\leq 0
$$
because $A-\omega\,\mathrm{id}$ is dissipative.
Since $A_\varepsilon$ is bounded on $L^2(\Omega,\R^d)$, we have $D(A_\varepsilon)=L^2(\Omega,\R^d)$, and thus its adjoint $A_\varepsilon^*$ is bounded and $D(A_\varepsilon^*)=L^2(\Omega,\R^d)$. Then, obviously, $A_\varepsilon^*-\omega\,\mathrm{id}$ is also dissipative.
The conclusion now follows from the Lumer-Phillips theorem (see \cite[Chapter II, Corollary 3.17]{EngelNagel} or \cite[Chapter 1, Theorem 4.3]{Pazy}).
\end{proof}

\begin{remark}
Lemma \ref{lem_Aeps_dissip} is the key step where we use the particular form $A_\varepsilon g=\eta_\varepsilon\star_\Omega A(\eta_\varepsilon\star_\Omega g)$, in order to ensure dissipativity. 
The dissipativity property is not satisfied if we choose $A_\varepsilon g=A(\eta_\varepsilon\star_\Omega g)$.
Note that, as already mentioned, \ref{H_stab} is always satisfied (with $M=1$) when choosing the Yosida approximant $A_\varepsilon = \left( \mathrm{id}-\varepsilon A \right)^{-1}A$. The interest of the above construction is that it is fully explicit.
\end{remark}

\paragraph{A first convergence property of $A_\varepsilon$.}

\begin{lemma}\label{lem_Aeps2}
Given any $g\in\mathscr{C}^\infty(\Omega,\R^d)$, we have $A_\varepsilon g \rightarrow Ag$ in $L^2(\Omega,\R^d)$ as $\varepsilon\rightarrow 0$.
\end{lemma}

\begin{proof}
By the triangular inequality, using the expression of $A_\varepsilon$ given by Lemma \ref{lem_Aeps},
$$
\Vert A_\varepsilon g-Ag\Vert_{L^2} \leq
\Vert \eta_\varepsilon\star_\Omega (A(\eta_\varepsilon\star_\Omega g) - Ag)\Vert_{L^2} + \Vert \eta_\varepsilon\star_\Omega Ag-Ag\Vert_{L^2} .
$$
The second term at the right-hand side of that inequality converges to $0$ as $\varepsilon\rightarrow 0$, because $Ag\in L^2(\Omega)$. 
To handle the first term, we use the Young inequality $\Vert F\star G\Vert_{L^r(\R^n)} \leq B_{p,q} \Vert F\Vert_{L^p(\R^n)} \Vert G\Vert_{L^q(\R^n)}$ with $F=\eta_\varepsilon$, $G=\tilde r_\varepsilon$ where $r_\varepsilon=A(\eta_\varepsilon\star_\Omega g)-Ag$, and with $r=2$, $p=1$ and $q=2$, obtaining 
$$
\Vert \eta_\varepsilon\star_\Omega r_\varepsilon\Vert_{L^2} 
= \Vert (\eta_\varepsilon\star \tilde r_\varepsilon)_{\vert\Omega}\Vert_{L^2(\R^n,\R^d)} 
\leq \Vert \eta_\varepsilon\star \tilde r_\varepsilon\Vert_{L^2(\R^n)} 
\leq B_{1,2} \Vert \eta_\varepsilon\Vert_{L^1(\Omega)} \Vert \tilde r_\varepsilon\Vert_{L^2(\R^n,\R^d)} 
= B_{1,2} \Vert r_\varepsilon\Vert_{L^2} 
$$
because $\Vert \eta_\varepsilon\Vert_{L^1(\Omega)}=1$, and we conclude that $A_\varepsilon g \rightarrow Ag$ in $L^2(\Omega)$ by noticing that $r_\varepsilon\rightarrow 0$ in $L^2(\Omega,\R^d)$ because
$$
A(\eta_\varepsilon\star_\Omega g)=\sum_{\vert\alpha\vert\leq p}a_\alpha D^\alpha(\eta_\varepsilon\star_\Omega g)=\sum_{\vert\alpha\vert\leq p}a_\alpha \, \eta_\varepsilon\star_\Omega D^\alpha g 
\longrightarrow \sum_{\vert\alpha\vert\leq p}a_\alpha D^\alpha g = Ag
$$
in $L^2(\Omega,\R^d)$ as $\varepsilon\rightarrow 0$. 
\end{proof}

In terms of Schwartz kernels, the kernel of $A_\varepsilon$ is obtained by convoluting to the left and ``to the right" (in some sense) the Schwartz kernel of $A$ with $\eta_\varepsilon$, and that, for every $x\in\Omega$ fixed, the $\R^{d\times d}$-valued function $x'\mapsto\sigma_\varepsilon(x,x')$ converges in the distributional sense to the $\R^{d\times d}$-valued distribution 
$\sum_{\vert\alpha\vert\leq p} (-1)^{\vert\alpha\vert} a_\alpha \delta_x^{(\alpha)}$ as $\varepsilon\rightarrow 0$ (where $\delta_x^{(\alpha)}$ is a distributional derivative of the Dirac $\delta_x$ at $x$), which is the Schwartz kernel of $A$.

The convergence property stated in Lemma \ref{lem_Aeps2} is not enough to get \ref{H_Aepsilon_CV}: we need to refine the analysis in order to get convergence estimates.
We start by refining our analysis of the unusual convolution operator introduced previously.

\paragraph{Convergence estimates for the convolution operator $\eta_\varepsilon\star_\Omega$.}

We introduce the following notation: in the case \ref{O1}, for any $\varepsilon\in(0,1]$ we define the compact subset $\Omega_\varepsilon$ of the interior of $\Omega$ by
\begin{equation*}
\Omega_\varepsilon = \{ x\in\Omega\ \mid\ \mathrm{d}_\Omega(x,\partial\Omega)\geq\varepsilon \} .
\end{equation*}
There exists a constant $C_{\partial\Omega}>0$ such that 
\begin{equation*}
\nu(\Omega\setminus\Omega_\varepsilon) \leq C_{\partial\Omega} \, \varepsilon^n \qquad \forall \varepsilon\in(0,1] .
\end{equation*}
In the case \ref{O2} we simply set $\Omega_\varepsilon=\Omega$ and $C_{\partial\Omega}=0$.

We will also need to use extension operators in the case \ref{O1}:
according to \cite[Chap. VI, Sec. 3, Theorem 5]{Stein_1970} (see also \cite[Chap. 12]{Tartar}), there exist $C_E>0$ and a linear continuous operator $E$ mapping functions on $\Omega$ to functions defined on the whole $\R^n$, such that the restriction of $Eg$ to $\Omega$ coincides with $g$ and $\Vert Eg\Vert_{W^{j,k}(\R^n)} \leq C_E\Vert g\Vert_{W^{j,k}(\Omega)}$ for every $g\in W^{j,k}(\Omega)$ and for every $j\in\N$ and every $k\in[1,+\infty]$ (Stein extension). 
In the case \ref{O2}, accordingly, we set $C_E=1$. Note that this is \emph{not} the extension by zero.

\begin{lemma}\label{lem_speed_CV_dirac}
Given any $\varepsilon\in(0,1]$, we have
$$
\Vert \eta_\varepsilon\star_\Omega g - g\Vert_{L^\infty} \leq 2\Vert g\Vert_{L^\infty} \qquad \forall g\in L^\infty(\Omega,\R^d) ,
$$
$$
\vert \eta_\varepsilon\star_\Omega g(x) - g(x)\vert \leq  C_E C_\eta \varepsilon \Vert g\Vert_{W^{1,\infty}} \qquad\forall x\in\Omega_\varepsilon\qquad \forall g\in W^{1,\infty}(\Omega,\R^d). 
$$
\end{lemma}

\begin{proof}
The first inequality is obviously obtained by using that $\Vert \eta_\varepsilon\star \tilde g\Vert_{L^\infty(\R^n)} \leq \Vert  \tilde g\Vert_{L^\infty(\R^n,\R^d)}$ because $\Vert \eta_\varepsilon\Vert_{L^1(\R^n,\R^d)}=1$.

Let $x\in\Omega_\varepsilon$ be arbitrary. 
Since $\supp(\eta_\varepsilon)\subset\overline B_{\R^n}(0,\varepsilon)$ and thus $\supp(\eta_\varepsilon(x-\cdot))\subset\overline B_{\R^n}(x,\varepsilon)\subset\Omega$, we have $\int_\Omega\eta_\varepsilon(x-x')\, dx' = \int_{\R^n}\eta_\varepsilon(x-x')\, dx' = 1$, hence $g(x) = \int_\Omega\eta_\varepsilon(x-x') g(x)\, dx'$ and 
$$
\vert\eta_\varepsilon\star g(x) - g(x)\vert 
\leq \int_\Omega \eta_\varepsilon(x-x') \vert g(x')- g(x)\vert\, dx' 
\leq \int_{\R^n} \eta_\varepsilon(x-x') \vert Eg(x')- Eg(x)\vert\, dx' 
$$
where $Eg$ is the Stein extension of $g$ (actually the latter inequality is even an equality because $\eta_\varepsilon(x-x')=0$ for any $x'\in\R^n\setminus\Omega$, since $x\in\Omega_\varepsilon$).
It follows from the mean value theorem that 
$$
\vert Eg(x')- Eg(x)\vert \leq \Vert Eg\Vert_{W^{1,\infty}(\R^n,\R^d)} \Vert x-x'\Vert
\leq C_E \Vert g\Vert_{W^{1,\infty}} \Vert x-x'\Vert .
$$
Hence
$$
\vert\eta_\varepsilon\star g(x) - g(x)\vert \leq C_E \Vert g\Vert_{W^{1,\infty}} \int_{\R^n} \frac{1}{\varepsilon^n}\eta\left(\frac{x-x'}{\varepsilon}\right) \Vert x-x'\Vert\, dx' = C_E C_\eta \varepsilon \Vert g\Vert_{W^{1,\infty}}
$$
by using the change of variable $x'=x-\varepsilon s$.

Note that, in the above argument, we have used a $W^{1,\infty}$ extension of $g$ (and not the extension by $0$, which is not in $W^{1,\infty}(\R^n,\R^d)$) in order to use the mean value theorem, because, in the case \ref{O1}, $\Omega$ may not be convex.
\end{proof}

\paragraph{Convergence properties of $A_\varepsilon$.}
We are now in a position to establish \ref{H_Aepsilon_CV}.
Recall that, by assumption, $\max_{\vert\alpha\vert\leq p} \Vert a_\alpha\Vert_{W^{1,\infty}}\leq C_a$.

\begin{lemma}\label{lem_CV_Aeps}
Given any $\varepsilon\in(0,1]$, we have
\begin{equation}\label{lem_CV_Aeps_1}
\Vert (A_\varepsilon-A)g\Vert_{L^\infty} \leq 4n^{p+1} C_a \Vert g\Vert_{W^{p,\infty}} \qquad \forall g\in W^{p,\infty}(\Omega,\R^d) ,
\end{equation}
\begin{equation}\label{lem_CV_Aeps_2}
\vert (A_\varepsilon-A)g (x) \vert \leq 3n^{p+1} C_E C_\eta C_a \varepsilon \Vert g\Vert_{W^{p+1,\infty}}\qquad \forall x\in\Omega_\varepsilon\qquad \forall g\in W^{p+1,\infty}(\Omega,\R^d) .
\end{equation}
As a consequence,
\begin{equation}\label{lem_CV_Aeps_3}
\Vert (A_\varepsilon-A)g\Vert_{L^2} \leq 3 n^{p+1} C_a \sqrt{C_E^2 C_\eta^2 \varepsilon^2 + 4 C_{\partial\Omega} \varepsilon^n} \Vert g\Vert_{W^{p+1,\infty}} \qquad \forall g\in W^{p+1,\infty}(\Omega,\R^d) ,
\end{equation}
and therefore \ref{H_Aepsilon_CV} is satisfied with $C_A=2n^{p+1}C_a$ and $\chi_A(\varepsilon) = \sqrt{C_E^2 C_\eta^2 \varepsilon^2 + 4 C_{\partial\Omega} \varepsilon^n}$.
\end{lemma}

Note that, in the case \ref{O2}, we have $C_{\partial\Omega}=0$ and then the above estimate is in $\varepsilon$. Actually, in both cases \ref{O1} and \ref{O2} the estimate is in $\varepsilon$ as $\varepsilon\rightarrow 0$ except in the case \ref{O1} when moreover $n=1$, in which case the estimate is in $\sqrt{\varepsilon}$.

\begin{proof}
For any $\alpha\in\N^k$ such that $\vert\alpha\vert\leq p$, noting that $D^\alpha(\eta_\varepsilon\star_\Omega g-g) = \eta_\varepsilon\star_\Omega D^\alpha g - D^\alpha g$, we infer from Lemma \ref{lem_speed_CV_dirac} applied to $D^\alpha g$ that 
$$
\Vert D^\alpha(\eta_\varepsilon\star_\Omega g-g)\Vert_{L^\infty} \leq 2\Vert f\Vert_{W^{p,\infty}}
$$
$$
\vert D^\alpha(\eta_\varepsilon\star_\Omega g-g)(x) \vert
= \vert \eta_\varepsilon\star_\Omega (D^\alpha g)(x)-D^\alpha g(x)\vert\leq C_E C_\eta \varepsilon \Vert g\Vert_{W^{\vert\alpha\vert+1,\infty}} \qquad \forall x\in\Omega_\varepsilon 
$$
and thus, using that $A=\sum_{\vert\alpha\vert\leq p} a_\alpha D^\alpha$ and that $\Vert a_\alpha\Vert_{L^\infty(\Omega)}\leq C_a$, and since the number of $\alpha\in\N^n$ such that $\vert\alpha\vert\leq p$ is $1+n+\cdots+n^p=\frac{n^{p+1}-1}{n-1}\leq n^{p+1}$, we obtain
$$
\Vert A(\eta_\varepsilon\star_\Omega g-g) \Vert_{L^\infty(\Omega)} \leq 2n^{p+1} C_a \Vert g\Vert_{W^{p,\infty}}
$$
$$
\vert A(\eta_\varepsilon\star_\Omega g-g)(x)\vert \leq n^{p+1} C_E C_\eta C_a \varepsilon\Vert g\Vert_{W^{p+1,\infty}}  \qquad \forall x\in\Omega_\varepsilon  .
$$
Besides, we infer from Lemma \ref{lem_speed_CV_dirac} applied to $A(\eta_\varepsilon\star_\Omega g)$, using that $\Vert a_\alpha\Vert_{W^{1,\infty}}\leq C_a$, that
\begin{multline*}
\Vert \eta_\varepsilon \star_\Omega A(\eta_\varepsilon\star_\Omega g) - A(\eta_\varepsilon\star_\Omega g) \Vert_{L^\infty} \leq 2 \Vert A(\eta_\varepsilon\star_\Omega g) \Vert_{L^\infty} \\
\leq 2n^{p+1} C_a \Vert \eta_\varepsilon\star_\Omega g\Vert_{W^{p,\infty}}
\leq 2n^{p+1} C_a \Vert g\Vert_{W^{p,\infty}}
\end{multline*}
and
\begin{multline*}
\vert \eta_\varepsilon \star_\Omega A(\eta_\varepsilon\star_\Omega g)(x) - A(\eta_\varepsilon\star_\Omega g)(x) \vert \leq C_E C_\eta \varepsilon \Vert A(\eta_\varepsilon\star_\Omega g) \Vert_{W^{1,\infty}}  \\
\leq 2 n^{p+1} C_E C_\eta C_a \varepsilon \Vert \eta_\varepsilon\star_\Omega g \Vert_{W^{p+1,\infty}}
\leq 2 n^{p+1} C_E C_\eta C_a \varepsilon \Vert g\Vert_{W^{p+1,\infty}} \qquad \forall x\in\Omega_\varepsilon
\end{multline*}
where we have used that $\Vert \eta_\varepsilon\star h\Vert_{L^\infty(\R^n)} \leq \Vert h\Vert_{L^\infty(\R^n)}$ for any $h\in L^\infty(\R^n)$ (recall that $\Vert\eta_\varepsilon\Vert_{L^1(\R^n)}=1$).
Finally, by the triangular inequality, we have
$$
\vert (A_\varepsilon-A)g(x) \vert \leq \vert \eta_\varepsilon \star_\Omega A(\eta_\varepsilon\star_\Omega g)(x) - A(\eta_\varepsilon\star_\Omega g) (x) \vert + \vert A(\eta_\varepsilon\star_\Omega g-g) (x)\vert
$$
and the estimates \eqref{lem_CV_Aeps_1} and \eqref{lem_CV_Aeps_2} follow.

To establish \eqref{lem_CV_Aeps_3}, we write
$$
\Vert (A_\varepsilon-A)g\Vert_{L^2}^2 = \int_{\Omega_\varepsilon} \vert (A_\varepsilon-A)g(x)\vert^2\, dx +\int_{\Omega\setminus\Omega_\varepsilon} \vert (A_\varepsilon-A)g(x)\vert^2\, dx .
$$
Using \eqref{lem_CV_Aeps_2}, the first term is estimated by
$$
\int_{\Omega_\varepsilon} \vert (A_\varepsilon-A)g(x)\vert^2\, dx \leq 9 n^{2(p+1)} C_E^2 C_\eta^2 C_a^2 \varepsilon^2 \Vert g\Vert_{W^{p+1,\infty}}^2 ,
$$
and using \eqref{lem_CV_Aeps_1}, the second term is estimated by
$$
\int_{\Omega\setminus\Omega_\varepsilon} \vert (A_\varepsilon-A)g(x)\vert^2\, dx \leq \Vert (A_\varepsilon-A)g\Vert_{L^\infty}^2 \nu(\Omega\setminus\Omega_\varepsilon)
\leq 36 n^{2(p+1)} C_a^2 \Vert g\Vert_{W^{p,\infty}}^2 C_{\partial\Omega}\, \varepsilon^n ,
$$
and the conclusion follows.
\end{proof}

\section{Examples}\label{examples}
In this section we state  the expression of the particle systems approximating a few emblematic quasilinear PDEs.

\subsection{Transport equations}
At a pedagogical level let us consider the most simple transport equation
\begin{equation}\nonumber
\partial_ty(t,x)+\partial_xy(t,x)=0
\end{equation}

The approximating particle systems reads
$$
\dot\xi_{\varepsilon,i}^N(t) = \frac{1}{N}\sum_{j=1}^N \eta_\varepsilon^{*2}(x_i^N-x_j^N)
\xi_{\varepsilon,j}^N(t) 
$$
for every $i\in\{1,\ldots,N\}$, where
$$
\eta^{*2}(x)=
\int_\Omega \eta_\varepsilon(x-x'') D\eta_\varepsilon(x'')dx''.
$$
Already on this example one sees the difference between the method of approximation developed in the present article with a simple discretization one: in our case all the values of $x^N_j$ are involved in the equation satisfied by $
\xi_{\varepsilon,i}^N(t) $, unlike the nearest neighbour's ones  for a simple discretization.  Note however that the limit $\varepsilon\to 0$ selects the values $x^N_j\sim x^N_i$.
\subsection{Burgers equation}
The Burgers equation reads
\begin{equation}\nonumber
\partial_ty(t,x)+y(t,x)\partial_xy(t,x)=0
\end{equation}
The associated particle systems is then

$$
\dot\xi_{\varepsilon,i}^N(t) = \frac{1}{N}\sum_{j,k=1}^N \eta_\varepsilon^{\Omega_k}(x_i^N,x_j^N)
\xi_{\varepsilon,k}^N(t) 
\xi_{\varepsilon,j}^N(t) 
$$
for every $i\in\{1,\ldots,N\}$, where
$$
\eta_\varepsilon^{\Omega_k}(x_i^N,x_j^N)=
\int_{\Omega_k} \eta_\varepsilon(x_i^N-x'') D\eta_\varepsilon(x''-x_j^N)dx''.
$$

\subsection{KdV equation}
The Korteweg–De Vries equation reads
\begin{equation}\nonumber
\partial_ty(t,x)+\partial^3_xy(t,x)+6y(t,x)\partial_xy(t,x)=0
\end{equation}
Therefore the system of particle approximating KdV is

$$
\dot\xi_{\varepsilon,i}^N(t) = 6\frac{1}{N}\sum_{j=1}^N D^2\eta_\varepsilon^{*2}(x_i^N-x_j^N)
\xi_{\varepsilon,j}^N(t) 
+
\frac{1}{N}\sum_{j,k=1}^N \eta_\varepsilon^{\Omega_k}(x_i^N,x_j^N)
\xi_{\varepsilon,k}^N(t) 
\xi_{\varepsilon,j}^N(t) 
$$
for every $i\in\{1,\ldots,N\}$, where again
$$
\eta^{*2}(x)=
\int_\Omega \eta_\varepsilon(x-x'') D\eta_\varepsilon(x'')dx''
$$
and
$$
\eta_\varepsilon^{\Omega_k}(x_i^N,x_j^N)=
\int_{\Omega_k} \eta_\varepsilon(x_i^N-x'') D\eta_\varepsilon(x''-x_j^N)dx''.
$$
\section{Further remarks}\label{sec_without_semigroup}
In this section, we show that the particle approximation result stated in Theorem \ref{thm_approx_abstract} can be extended to some cases where the operator does even not generate a semigroup, like the case of the backward heat equation $\partial_t y = -\triangle y$, {and even to some nonlinear cases}.

We have seen in Section \ref{sec_PDE} that the strategy to approximate a given solution $y$ to a quasilinear PDE  goes in two steps: first, find an adequate  bounded approximation $A_\varepsilon$ of $A$, and $y_\varepsilon$ of $y$; second, take the particle approximation $y_\varepsilon^N$ of $y_\varepsilon$.
The second step is an automatic consequence of Theorem \ref{thm_estim_graph}  and is thus general. 
The first step has been performed in Sections \ref{sec_abstract} and \ref{sec_PDE} by applying the Duhamel formula, within the semigroup context, which required the instrumental uniform stability estimate \ref{H_sigmaepsilon} this is in such a way that, in the first step of the proof of Theorem \ref{thm_approx_abstract}, we have established the inequality \eqref{CV_yeps_abstract}, i.e., 
\begin{equation}\label{inegepsi}
\Vert y_\varepsilon(t)-y(t)\Vert_X \leq C \chi(\varepsilon) \Vert y\Vert_{L^1([0,T],Z)} \qquad\forall t\in[0,T]
\end{equation}
for some $C>0$. 
But this first step, requiring the demanding estimate \ref{sec_PDE}, can be dropped if one is able to design a bounded approximation $A_\varepsilon$ of $A$ and an approximation $y_\varepsilon$ of $y$ such that the estimate \eqref{inegepsi} is satisfied. 
And indeed this can often be done, without requiring any semigroup property. 
Let us give some examples. 

\paragraph{Backward heat equation.} Consider the backward heat equation $\partial_t y = -\triangle y$ and its approximation $\partial_t y_\varepsilon = -\triangle_\varepsilon y_\varepsilon$ where $\triangle_\varepsilon$ is a bounded approximation of $\triangle$ as done in the previous sections. 
Assuming that $y(0)=e^{t\triangle}f$ for some $g\in L^2(\Omega,\R^d)$, we have $y(t) = e^{(T-t)\triangle}f$ for every $t\in[0,T]$. Now, we take $y_\varepsilon(0)=e^{t\triangle_\varepsilon}f$ and we have as well $y_\varepsilon(t) = e^{(T-t)\triangle_\varepsilon}f$ for every $t\in[0,T]$. Then, obviously, the Duhamel formula gives \eqref{inegepsi}. 

Of course, this works because we have considered a very regular initial condition.
More generally, the argument works for operators with constant coefficients, taking Fourier transforms and considering initial conditions whose Fourier transform has a compact support. 

\paragraph{Variational inequalities.} There exists a wide existing literature on variational inequalities, with the objective of establishing the existence of a solution to a nonlinear equation $\partial_t y=A(y)$ by approximating the nonlinear unbounded operator $A$ with a bounded operator $A_\varepsilon$. The estimate \eqref{inegepsi} can then obtained from energy considerations, or from Galerkin approximation considerations, etc. 
Most of known equations having a physical meaning enter in such a framework.

\paragraph{Acknowledgment.}
We are indebted to Claude Bardos, Julien Barr\'e, Arnaud Debussche, Nicolas Fournier, Isabelle Gallagher, Thierry Gallay, Alain Joye, Beno\^{\i}t Perthame and Laure Saint-Raymond for useful discussions.

\appendix

\section{Appendix: Finite particle approximation of quasilinear integral evolution equations}\label{app_abstract_quasilinear_integral}
As at the beginning of Section \ref{sec_abstract}, let $(\Omega,\mathrm{d}_\Omega)$ be a complete metric space, endowed with a probability measure $\nu\in\mathcal{P}(\Omega)$, having a family $(\mathcal{A}^N,x^N)_{N\in\N^*}$ of tagged partitions associated with $\nu$, satisfying \eqref{def_tagged}.
Hereafter, the set $\R^{d\times d}$ of square real-valued matrices of size $d$ is equipped with the matrix norm $\Vert\ \Vert_{\R^{d\times d}}$ induced by the Euclidean norm $\Vert\ \Vert_{\R^d}$ on $\R^d$. 

\subsection{Quasilinear integral evolution equation}\label{app_abstract_quasilinear_integral_setting}
Let $T>0$.
Throughout this appendix, we consider the quasilinear integral evolution equation on $[0,T]$
\begin{equation}\label{abstract_quasilinear_integral}
\boxed{
\partial_t y(t,x) = \int_\Omega \sigma[t,y(t)](x,x') \, y(t,x')\, d\nu(x') + f[t,y(t)](x)
}
\end{equation}
where $\sigma[t,y(t)](x,x')\in\R^{d\times d}$ and $y(t)(x)=y(t,x)\in\R^d$.
We follow the framework of Section \ref{sec_abstract} with $X=Z=L^\infty(\Omega,\R^d)$ (sometimes denoted in short, in what follows, by $L^\infty$): the equation \eqref{abstract_quasilinear_integral} is of the form \eqref{abstract_quasilinear} with the \emph{bounded} linear operator $A[t,z]$ defined by 
$$
A[t,z]\, y(x) = \int_\Omega \sigma[t,z](x,x') \, y(x')\, d\nu(x') .
$$
Let $y^0\in L^\infty(\Omega,\R^d)$.
We assume that
there exists $r>0$ such that, for all $(t,z)\in[0,T]\times B_{L^\infty}(y^0,r)$:
\begin{enumerate}[label=$(A_{\arabic*})$]
\item\label{A_sigma} the kernel $\sigma[t,z](x,x')\in\R^{d\times d}$ depends continuously on $(t,z,x,x')\in [0,T]\times B_{L^\infty}(y^0,r)\times\Omega\times\Omega$ and is bounded and Lipschitz continuous with respect to $z$, uniformly with respect to $(t,x,x')$, i.e., 
\begin{equation*}
\begin{split}
\Vert\sigma[t,z](x,x'))\Vert_{\R^{d\times d}} &\leq \Vert\sigma\Vert_\infty , \\
\Vert\sigma[t,z_1](x,x')-\sigma[t,z_2](x,x')\Vert_{\R^{d\times d}} &\leq \Lip(\sigma) \Vert z_1-z_2\Vert_{L^\infty} ,
\end{split}
\end{equation*}
for all $t\in[0,T]$, $z_1,z_2\in B_{L^\infty}(y^0,r)$ and $x,x'\in\Omega$.
\item\label{A_f} 
$f[t,z](x)$ depends continuously on $(t,z,x)\in [0,T]\times B_{L^\infty}(y^0,r)\times\Omega$ and is bounded and Lipschitz continuous with respect to $z$, uniformly with respect to $(t,x)$, i.e., 
\begin{equation*}
\begin{split}
\Vert f[t,z](x)\Vert_{\R^d} &\leq \Vert f\Vert_\infty , \\
\Vert f[t,z_1]-f[t,z_2]\Vert_{L^\infty} &\leq \Lip(f) \Vert z_1-z_2\Vert_{L^\infty} ,
\end{split}
\end{equation*}
for all $t\in[0,T]$, $z_1,z_2\in B_{L^\infty}(y^0,r)$ and $x\in\Omega$.
\end{enumerate}
It is then easy to see that Assumptions \ref{A_sigma} and \ref{A_f} imply Assumptions \ref{H_Z} to \ref{H_f_Lip} of Section \ref{sec_abstract} with $X=Z=L^\infty(\Omega,\R^d)$, $S=\mathrm{id}$, $C_4=\Lip(\sigma)$, $B=0$, $C_6=\Vert f\Vert_\infty$, $C_7=\Lip(f)$. 

By Proposition \ref{prop_existence_uniqueness}, there exists a unique solution $y(\cdot)\in\mathscr{C}^1([0,T'],L^\infty(\Omega,\R^d))$ of \eqref{abstract_quasilinear_integral} such that $y(0)=y^0$, for some $T'\in(0,T]$. Moreover, we have $y(t)\in B_{L^\infty}(y^0,r)$ for every $t\in[0,T']$. 
We recall and insist on the fact that the time $T'$ depends, in particular, on the choice of the space $L^\infty(\Omega,\R^d)$.

\subsection{Finite particle approximation}\label{app_prop_existence_yN}
Following the framework of Section \ref{sec_abstract_particle_approx}, 
we use the family $(\mathcal{A}^N,x^N)_{N\in\N^*}$ of tagged partitions of $\Omega$ associated with $\nu$, satisfying \eqref{def_tagged}, with $\mathcal{A}^N=(\Omega_1^N,\ldots,\Omega_N^N)$ and $x^N=(x_1^N,\ldots,x_N^N)$,
and we consider the finite particle system that is naturally associated with \eqref{abstract_quasilinear_integral}, given for any $N\in\N^*$ by
\begin{equation}\label{particle_system}
\boxed{
\dot\xi_i^N(t) = \frac{1}{N}\sum_{j=1}^N \sigma [t, y^N(t)] (x_i^N,x_j^N) \, \xi_j^N(t) + f[t,y^N(t)](x_i^N)
\qquad \forall i\in\{1,\ldots,N\} 
}
\end{equation}
where $y^N(t)\in L^\infty(\Omega,\R^d)$ is the piecewise constant function on $\Omega$ defined by
\begin{equation}\label{defyN}
\boxed{
y^N(t)(x) = y^N(t,x) = \sum_{i=1}^N \xi_i^N(t) \, \mathds{1}_{\Omega_i^N}(x) 
}
\end{equation}
In particular, we have $y^N(t,x_i^N)=\xi_i^N(t)$ for every $i\in\{1,\ldots,N\}$. 
Thanks to \eqref{defyN}, the particle system \eqref{particle_system} is equivalent to the evolution equation
\begin{equation}\label{eq_yN}
\boxed{
\partial_t y^N(t,x) = \int_\Omega \sigma^N [t, y^N(t)] (x,x') \, y^N(t,x')\, d\nu(x') + f^N[t,y^N(t)](x)
}
\end{equation}
where $\sigma^N [t, y^N(t)]$ and $f^N[t,y^N(t)]$ are also piecewise constant: for any $z\in L^\infty(\Omega,\R^d)$, we set
$\sigma^N[t,z](x,x') = \sigma[t,z](x_i^N,x_j^N)$
and
$f^N[t,z](x) = f[t,z](x_i^N)$ if $x\in\Omega_i^N$ and $x'\in\Omega_j^N$. 
Solutions $t\mapsto(\xi_1^N(t),\ldots,\xi_N^N(t))\in(\R^d)^N$ of the particle system \eqref{particle_system} are indifferently seen as solutions $t\mapsto y^N(t)$ of the quasilinear integral evolution equation \eqref{eq_yN} and conversely. 


\begin{proposition}\label{prop_existence_yN}
Assume that $y^0\in\mathscr{C}^0(\Omega,\R^d)$.
For every $N\in\N^*$, there exists a unique solution $t\mapsto(\xi_1^N(t),\ldots,\xi_N^N(t))\in(\R^d)^N$ of \eqref{particle_system} such that $\xi_i^N(0)=y^0(x_i^N)$ for every $i\in\{1,\ldots,N\}$, well defined on the time interval $[0,T'']$ if $N\geq N_0$, for some $T''\in(0,T']$ and some $N_0\in\N^*$; moreover, $y^N(t)\in B_{L^\infty}(y^0,r)$ for every $t\in[0,T'']$ and every $N\geq N_0$.
\end{proposition}

\begin{proof}
The nontrivial statement is the fact that the time interval of definition is uniform with respect to sufficiently large values of $N$. This is obtained by applying Proposition \ref{prop_existence_uniqueness} and Remark \ref{rem_prop_existence_uniqueness} to the evolution equation \eqref{eq_yN} with $X=Z=L^\infty(\Omega,\R^d)$ as in Appendix \ref{app_abstract_quasilinear_integral_setting}, but with the operator $A^N[t,z]$ defined by $A^N[t,z]\, y(x) = \int_\Omega \sigma^N[t,z](x,x') \, y(x')\, d\nu(x')$ and with the function $f^N[t,z]$, noting that they satisfy Assumptions \ref{H_stab} to \ref{H_f_Lip} uniformly with respect to $N$.
The initial condition for \eqref{eq_yN} is $y^N(0) = \sum_{i=1}^N y^0(x_i^N) \, \mathds{1}_{\Omega_i^N}$.
Since $y^0$ is continuous on $\Omega$, the function $y^N(0)$ converges uniformly to $y^0$ (i.e., in $L^\infty$ topology) as $N\rightarrow+\infty$, hence $y^N(0)\in B_{L^\infty}(y^0,r/2)$ if $N$ is large enough. This proves the claim.
\end{proof}

\begin{remark}
In Proposition \ref{prop_existence_yN}, we have assumed that $y^0$ is continuous in order to ensure that, for $N$ large enough, $y^N(0)$ is sufficiently close to $y^0$ in $L^\infty$. An alternative, if one does not want to make this additional assumption, is to strengthen Assumptions \ref{A_sigma} and \ref{A_f} (and thus, Assumptions \ref{H_stab} to \ref{H_f_Lip}) by assuming that they are satisfied for every $r>0$ (while we have assumed that they are satisfied for one given $r$). Indeed, then, given any $y^0\in L^\infty(\Omega,\R^d)$, one chooses $r>0$ sufficiently large so that $y^N(0)\in B_{L^\infty}(y^0,r/2)$: we obtain Proposition \ref{prop_existence_yN} as well, with solutions defined on $[0,T'']$ for every $N\in\N^*$ (not only for $N$ large enough). 

In \ref{H_stab} to \ref{H_f_Lip}, we have assumed that ``there exists $r>0$", to follow the classical references \cite{Kato_1975, Kato_1993, Pazy}. But given that, in this appendix, we take $X=Z=L^\infty(\Omega,\R^d)$, in practice for most examples the assumptions are also satisfied for every $r>0$.
\end{remark}

\subsection{Convergence result}\label{app_thm_estim_graph}

\begin{theorem}\label{thm_estim_graph}
In the framework of Proposition \ref{prop_existence_yN}:
\begin{enumerate}[label=(\roman*), leftmargin=*,parsep=0mm, itemsep=0mm, topsep=1mm]
\item\label{item_CV_y_yN} For every $t\in[0,T'']$, we have $y(t,\cdot)\in\mathscr{C}^0(\Omega,\R^d)$ and
\begin{equation}\label{CV_y_yN}
\lim_{N\rightarrow+\infty} y^N(t,x) = y(t,x)
\end{equation}
uniformly with respect to $(t,x)\in[0,T'']\times\Omega$. In particular,
\begin{equation*}
\lim_{N\rightarrow+\infty} \xi_i^N(t) = y(t,x_i^N)\qquad\forall i\in\{1,\ldots,N\}.
\end{equation*}
\item\label{item_estim_y_yN} We assume moreover that $y^0$ is Lipschitz continuous on $\Omega$, with a Lipschitz constant $\Lip(y^0)$, and that, additionally to \ref{A_sigma} and \ref{A_f}, the functions $\sigma[t,z]$ and $f[t,z]$ are also Lipschitz continuous, i.e., we assume that
\begin{equation*}
\begin{split}
\Vert\sigma[t,z_1](x_1,x'_1)-\sigma[t,z_2](x_2,x'_2)\Vert_{\R^{d\times d}} 
&\leq \Lip(\sigma)\left( \Vert z_1-z_2\Vert_{L^\infty} + \mathrm{d}_\Omega(x_1,x_2) + \mathrm{d}_\Omega(x'_1,x'_2) \right) , \\
\Vert f[t,z_1](x_1)-f[t,z_2](x_2)\Vert_{\R^d} 
&\leq \Lip(f)\left( \Vert z_1-z_2\Vert_{L^\infty} + \mathrm{d}_\Omega(x_1,x_2) \right) ,
\end{split}
\end{equation*}
for all $t\in[0,T]$, $z_1,z_2\in B_{L^\infty}(y^0,r)$ and $x_1,x_2,x'_1,x'_2\in\Omega$.
Then, for every $t\in[0,T'']$, the function $y(t,\cdot)$ is Lipschitz continuous on $\Omega$, and
\begin{equation}\label{estim_y_yN}
\Vert y(t,x)-y^N(t,x)\Vert_{\R^d} \leq \frac{1}{N^\gamma} (a_1+a_2t) e^{a_3t} \leq \frac{\mathrm{Cst}}{N^\gamma}
\end{equation}
for all $(t,x)\in[0,T'']\times\Omega$ and $N\geq N_0$, with 
$$
a_1 = C_\Omega\Lip(y^0) + \frac{C_\Omega}{a_3} \left( 2\Lip(\sigma)(r+\Vert y^0\Vert_{L^\infty}) + \Lip(f) +\Vert\sigma\Vert_\infty \Lip(y^0) \right) ,
$$
$$
a_2 = \frac{\Vert\sigma\Vert_\infty}{a_3} \left(\Lip(\sigma) r + \Lip(f) \right),
\qquad
a_3=\Lip(\sigma)(r+\Vert y^0\Vert_{L^\infty}) + \Vert\sigma\Vert_\infty + \Lip(f) .
$$
In particular,
\begin{equation*}
\Vert y(t,x_i^N)-\xi_i^N(t)\Vert_{\R^d} \leq \frac{\mathrm{Cst}}{N^\gamma} .
\end{equation*}
\item\label{item_prolong_yN} In addition to \ref{item_estim_y_yN}, we assume that $y(\cdot)$ is well defined on the whole interval $[0,T]$ (recall that $T''\leq T'\leq T$) and that $y(t)\in B_{L^\infty}(y^0,r/2)$ for every $t\in[0,T]$. 
Then, there exists $N_1\geq N_0$ such that, for every $N\geq N_1$, $y^N(\cdot)$ is well defined on the whole interval $[0,T]$ and $y^N(t)\in B_{L^\infty}(y^0,r/4)$ for every $t\in[0,T]$. 
Moreover, the estimate \eqref{estim_y_yN} is valid on $[0,T]$.
\end{enumerate}
\end{theorem}

Items \ref{item_CV_y_yN} and \ref{item_estim_y_yN} of Theorem \ref{thm_estim_graph} are reminiscent of the graph limit convergence result stated in \cite[Theorem 2.2]{PaulTrelat}. 

\begin{proof}
We start by proving \ref{item_estim_y_yN}. Hence, we assume that $\sigma$ and $f$ are Lipschitz continuous with respect to $z$ and $x$.

\begin{lemma}\label{lem_y_Lip}
We have $y(t)\in\Lip(\Omega,\R^d)$ 
for every $t\in[0,T']$, and $\Lip(y(t)) = \Lip(y^0) + t(\Lip(\sigma) r + \Lip(f))$.
\end{lemma}

\begin{proof}[Proof of Lemma \ref{lem_y_Lip}.]
Since $y(t)\in B_{L^\infty}(y^0,r)$, $\sigma[t,y(t)]$ and $f[t,y(t)]$ are Lipschitz continuous with respect to $x$ and thus, using \eqref{abstract_quasilinear_integral}, we have
\begin{equation*}
\begin{split}
\Vert \partial_t y(t,x)-\partial_t y(t,x')\Vert_{\R^d} 
&\leq \int_\Omega \Vert\sigma[t,y(t)](x,x'')-\sigma[t,y(t)](x',x'')\Vert_{\R^{d\times d}} \Vert y(t,x'')\Vert_{\R^d}\, d\nu(x'') \\
&\qquad\qquad + \Vert f[t,y(t)](x)-f[t,y(t)](x')\Vert_{\R^d} \\
&\leq (\Lip(\sigma) r+\Lip(f)) \mathrm{d}_\Omega(x,x')
\end{split}
\end{equation*}
and by integration we infer that $\Vert y(t,x)-y(t,x')\Vert_{\R^d} \leq \Vert y^0(x)-y^0(x')\Vert_{\R^d} + t(\Lip(\sigma) r+\Lip(f)) \mathrm{d}_\Omega(x,x')$. The result follows. 
\end{proof}

Let us establish \eqref{estim_y_yN}. We set $r^N(t)(x)=r^N(t,x) = y(t,x) - y^N(t,x)$, for all $t\in[0,T']$ and $x\in\Omega$. By definition, given any $i\in\{1,\ldots,N\}$ and any $x\in\Omega_i^N$, we have
\begin{eqnarray}
\partial_t r^N(t,x) &=& 
\int_\Omega \sigma[t,y(t)](x,x') y(t,x')\, d\nu(x') - \int_\Omega \sigma[t,y(t)](x_i^N,x') y(t,x')\, d\nu(x') \label{rNdot1} \\
&&+ \int_\Omega \sigma[t,y(t)](x_i^N,x') y(t,x')\, d\nu(x') - \frac{1}{N} \sum_{i=1}^N \sigma[t,y(t)](x_i^N,x_j^N) y(t,x_j^N) \label{rNdot2} \\
&& + f[t,y(t)](x) - f[t,y(t)](x_i^N) \label{rNdot5} \\
&&+ \frac{1}{N} \sum_{i=1}^N \left( \sigma[t,y(t)](x_i^N,x_j^N) - \sigma[t,y^N(t)](x_i^N,x_j^N) \right) y(t,x_j^N) \label{rNdot3} \\
&&+ \frac{1}{N} \sum_{i=1}^N \sigma[t,y^N(t)](x_i^N,x_j^N) \left( y(t,x_j^N) - y^N(t,x_j^N) \right) \label{rNdot4} \\
&& + f[t,y(t)](x_i^N) - f[t,y^N(t)](x_i^N) \label{rNdot6}
\end{eqnarray}
Let us estimate the various terms \eqref{rNdot1} to \eqref{rNdot5} in $\R^d$ norm.
Since $y(t)\in B_{L^\infty}(y^0,r)$ (and thus, $\Vert y(t,\cdot)\Vert_{L^\infty}\leq r+\Vert y^0\Vert_{L^\infty}$):
\begin{itemize}
\item Using the Lipschitz property of $\sigma[t,y(t)]$ assumed in \ref{item_estim_y_yN}, the norm of the term \eqref{rNdot1} is bounded by $\Lip(\sigma)(r+\Vert y^0\Vert_{L^\infty}) \mathrm{d}_\Omega(x,x_i^N)$, which is less than $\Lip(\sigma)(r+\Vert y^0\Vert_{L^\infty}) \frac{C_\Omega}{N^\gamma}$ by using \eqref{def_tagged}.
\item $\sigma[t,y(t)]$ is bounded by $\Vert\sigma\Vert_\infty$ and is Lipschitz continuous by the assumption done in \ref{item_estim_y_yN}; besides, $y(t)$ is bounded by $r+\Vert y^0\Vert_{L^\infty}$ and is $\Lip(y(t))$-Lipschitz continuous on $\Omega$ by Lemma \ref{lem_y_Lip}. Therefore, the function $x'\mapsto\sigma[t,y(t)](x,x')y(t,x')$ is Lipschitz continuous on $\Omega$, uniformly with respect to $x\in\Omega$, of Lipschitz constant $L_{\sigma y(t)} = \Lip(\sigma) (r+\Vert y^0\Vert_{L^\infty})+\Vert\sigma\Vert_\infty \Lip(y(t))$.
Now, using the estimate \eqref{CV_Riemann_sum_3} (Riemann sum theorem), it follows that the norm of \eqref{rNdot2} is bounded by $\frac{L_{\sigma y(t)}C_\Omega}{N^\gamma}$.
\item Using the Lipschitz property of $f[t,y(t)]$ assumed in \ref{item_estim_y_yN}, the norm of \eqref{rNdot5} is bounded by $\Lip(f)\mathrm{d}_\Omega(x,x_i^N)\leq\Lip(f) \frac{C_\Omega}{N^\gamma}$ (by using \eqref{def_tagged}).
\end{itemize}
Since $y(t)\in B_{L^\infty}(y^0,r)$ and $y^N(t)\in B_{L^\infty}(y^0,r)$ for $N\geq N_0$ by Proposition \ref{prop_existence_yN} (and thus, $\Vert y(t,\cdot)\Vert_{L^\infty}\leq r+\Vert y^0\Vert_{L^\infty}$ and $\Vert y^N(t,\cdot)\Vert_{L^\infty}\leq r+\Vert y^0\Vert_{L^\infty}$):
\begin{itemize}
\item Using \ref{A_sigma}, we have $\Vert \sigma[t,y(t)](x_i^N,x_j^N) - \sigma[t,y^N(t)](x_i^N,x_j^N)\Vert_{\R^{d\times d}} \leq \Lip(\sigma)\Vert y(t)-y^N(t)\Vert_{L^\infty}$, and thus the norm of \eqref{rNdot3} is bounded by $\Lip(\sigma)(r+\Vert y^0\Vert_{L^\infty})\Vert y(t)-y^N(t)\Vert_{L^\infty}$.
\item $\sigma[t,y^N(t)](x_i^N,x_j^N)\leq \Vert\sigma\Vert_\infty$, hence the norm of \eqref{rNdot4} is bounded by $\Vert\sigma\Vert_\infty\Vert y(t)-y^N(t)\Vert_{L^\infty}$.
\item Using \ref{A_f}, the norm of \eqref{rNdot6} is bounded by $\Lip(f)\Vert y(t)-y^N(t)\Vert_{L^\infty}$.
\end{itemize}
Therefore, we obtain
$$
\Vert\dot r^N(t)\Vert_{L^\infty} \leq a \Vert r^N(t)\Vert_{L^\infty} + \frac{b(t)}{N^\gamma}  
$$
for every $t\in[0,T']$ and every $N\geq N_0$,
with $a = \Lip(\sigma)(r+\Vert y^0\Vert_{L^\infty}) + \Vert\sigma\Vert_\infty + \Lip(f)$ and $b(t) = 2C_\Omega \Lip(\sigma)(r+\Vert y^0\Vert_{L^\infty}) + C_\Omega\Lip(f) + C_\Omega\Vert\sigma\Vert_\infty \Lip(y(t))$.
Integrating, and noting that $b$ is nondecreasing, we get
$$
\Vert r^N(t)\Vert_{L^\infty} \leq e^{at}\Vert r^N(0)\Vert_{L^\infty} + \frac{b(t)}{aN^\gamma}(e^{at}-1)
$$
and we conclude by noting that, for every $x\in\Omega$, there exists $i\in\{1,\ldots,N\}$ such that $x\in\Omega_i^N$ and then, since $y^N(0) = \sum_{i=1}^N y^0(x_i^N) \, \mathds{1}_{\Omega_i^N}$, we have  $\Vert r^N(t,0)\Vert_{\R^d} = \Vert y^0(x)-y^0(x_i^N)\Vert_{\R^d} \leq \Lip(y^0) \mathrm{d}_\Omega(x,x_i^N) \leq \frac{\Lip(y^0)C_\Omega}{N^\gamma}$.
We have proved \eqref{estim_y_yN}.

\medskip

Let us now prove \ref{item_CV_y_yN}.
Starting as in the proof of Lemma \ref{lem_y_Lip}, we infer from the continuity of $\sigma[t,y(t)]$ and of $f[t,y(t)]$ that, for any $\varepsilon>0$, if $x$ and $x'$ are sufficiently close then
$\Vert\partial_t y(t,x)-\partial_t y(t,x')\Vert_{\R^d}\leq \varepsilon (r+1)$
and therefore $\Vert y(t,x)-y(t,x')\Vert_{\R^d}\leq \Vert y^0(x)-y^0(x')\Vert_{\R^d}+\varepsilon (r+1)t$.
As a consequence, $y(t)$ is continuous on $\Omega$.
To establish \eqref{CV_y_yN}, we estimate again the various terms above but instead of using Lipschitz properties which yield estimates, we just use continuity properties to get that \eqref{rNdot1} converges to $0$, and we use \eqref{CV_Riemann_sum_2} (Riemann sum theorem) to get that \eqref{rNdot2} converges to $0$. The term \eqref{rNdot5} converges to $0$ by continuity of $f[t,y(t)]$. The estimates for \eqref{rNdot3}, \eqref{rNdot4} and \eqref{rNdot6} are unchanged. We obtain $\Vert\dot r^N(t)\Vert_{L^\infty} \leq a \Vert r^N(t)\Vert_{L^\infty} + \mathrm{o}(1)$ as $N\rightarrow+\infty$, and the result follows, using that $\Vert r^N(0)\Vert_{L^\infty}\rightarrow 0$.

\medskip

Let us finally establish \ref{item_prolong_yN}. We follow the proof of \ref{item_estim_y_yN}. First, using the additional assumption that $y(t)\in B_{L^\infty}(y^0,r)$ for every $t\in[0,T]$, Lemma \ref{lem_y_Lip} is now valid on the whole interval $[0,T]$. We now write $\partial_t r^N(t,x)$ as the sum of the terms \eqref{rNdot1} to \eqref{rNdot6}. The estimates of \eqref{rNdot1}, \eqref{rNdot2} and \eqref{rNdot5} are unchanged and are valid on $[0,T]$, since only $y(t)$ is involved.
Getting the estimates of \eqref{rNdot3}, \eqref{rNdot4} and \eqref{rNdot6} for any $N\geq N_0$ requires an additional argument because we do not know yet that $y^N(t)\in B_{L^\infty}(y^0,r)$ for every $t\in[0,T]$. Given any fixed $N\geq N_0$, let $T_N\in[T',T]$ be the maximal time such that $y^N(t)\in B_{L^\infty}(y^0,r)$ for every $t\in[0,T_N]$. Then, the estimates of \eqref{rNdot3}, \eqref{rNdot4} and \eqref{rNdot6} are valid on the interval $[0,T_N]$, and then the estimate \eqref{estim_y_yN} is obtained as previously, on the time interval $[0,T_N]$. 
Since $T_N\leq T$, we thus have $\Vert y(t)-y^N(t)\Vert_{L^\infty} \leq \frac{1}{N^\gamma} (a_1+a_2T) e^{a_3T}$ for every $t\in[0,T_N]$. Choosing $N_1\in\N^*$ large enough so that $\frac{1}{N_1^\gamma} (a_1+a_2T) e^{a_3T}\leq \frac{r}{4}$, using that $y(t)\in B_{L^\infty}(y^0,r/2)$ and using the triangular inequality, we infer that $y^N(t)\in B_{L^\infty}(y^0,\frac{3r}{4})$ for every $t\in[0,T_N]$. This argument shows that $T_N=T$, and finishes the proof of \ref{item_prolong_yN}. 
\end{proof}

\end{document}